\documentclass[11pt]{amsart}

\usepackage{amssymb,pstricks,pst-node,enumerate,kbordermatrix}

\usepackage[ps2pdf, colorlinks]{hyperref}

\newtheorem{thm}{Theorem}[section]
\newtheorem*{thm1}{Theorem~\ref{thm1}}
\newtheorem{lem}[thm]{Lemma}
\newtheorem{cor}[thm]{Corollary}
\newtheorem{prop}[thm]{Proposition}
\newtheorem{claim}{Claim}[thm]
\newtheorem{assumption}[claim]{Assumption}
\newtheorem{sclaim}{Claim}[claim]

\theoremstyle{definition}
\newtheorem{dfn}[thm]{Definition}

\newenvironment{subproof}{\begin{proof}}{\end{proof}}

\newcommand{\mcal}[1]{\ensuremath{\mathcal{#1}}}
\newcommand{\ag}[1]{\ensuremath{\mathrm{AG}{(#1)}}}
\newcommand{\pg}[1]{\ensuremath{\mathrm{PG}{(#1)}}}
\newcommand{\rbox}[1]{\makebox[0.75em][r]{#1}}
\newcommand{\tup}[1]{\textup{#1}}

\newcommand{\del}{\hspace{-0.5pt}\backslash}
\newcommand{\con}{\ensuremath{\hspace{-0.5pt}/}}
\newcommand{\dy}{\ensuremath{\Delta \textrm{-}Y}}
\newcommand{\yd}{\ensuremath{Y\textrm{-}\Delta}}
\newcommand{\nfd}{\ensuremath{(F_{7}^{-})^{*}}}
\newcommand{\agde}{\ensuremath{\mathrm{AG}(2, 3)\hspace{-0.5pt}\backslash e}}
\newcommand{\qalpha}{\ensuremath{\mathbb{Q}(\alpha)}}
\newcommand{\dash}{\nobreakdash-\hspace{0pt}}
\newcommand{\Z}{ \mathbb{Z} }
\newcommand{\ring}{\mathbb{O}}
\newcommand{\field}{\ensuremath{\mathbb{F}}}
\newcommand{\group}{\mathbf{G}}
\newcommand{\parf}{\ensuremath{\mathbb{P}}}
\newcommand{\nreg}{\mathbb{U}_1}
\newcommand{\bip}{G}
\newcommand{\symdiff}{\triangle}
\newcommand{\rank}{\ensuremath{r}}
\newcommand{\corank}{\ensuremath{\rank^{*}}}
\newcommand{\matrank}{\ensuremath{\mathrm{rank}}}
\newcommand{\si}{\operatorname{si}}
\newcommand{\co}{\operatorname{co}}

\DeclareMathOperator{\GF}{GF}
\DeclareMathOperator{\nigh}{nigh}

\begin{document}

\title[The excluded minors for near-regular matroids]
{The excluded minors for near-regular matroids}

\author[Hall]{Rhiannon Hall}
\address{School of Information Systems, Computing and Mathematics,
Brunel University, Uxbridge UB8 3PH, United Kingdom}
\email{rhiannon.hall@brunel.ac.uk}
\thanks{The first author was supported by a Nuffield
Foundation Award for Newly Appointed Lecturers in Science,
Engineering and Mathematics.}

\author[Mayhew]{Dillon Mayhew}
\address{School of Mathematics, Statistics, and Operations Research,
Victoria University of Wellington,
New Zealand}
\email{dillon.mayhew@msor.vuw.ac.nz}
\thanks{The second author was supported by a Foundation for
Research Science \& Technology post-doctoral fellowship.}

\author[Van Zwam]{Stefan H. M. van Zwam}
\address{Faculteit der Wiskunde en Informatica,
Technische Universiteit Eindhoven,
The Netherlands}
\email{svzwam@win.tue.nl}
\thanks{The third author was supported by NWO, grant 613.000.561.}

\date{\today}

\begin{abstract}
In unpublished work, Geelen proved that a matroid is
near-regular if and only if it has no minor isomorphic
to $U_{2,5}$, $U_{3,5}$, $F_{7}$, $F_{7}^{*}$,
$F_{7}^{-}$, \nfd, \agde, $(\agde)^{*}$, $\Delta_{T}(\agde)$,
or $P_{8}$.
We provide a proof of this characterization.
\end{abstract}

\maketitle

\section{Introduction}\label{sct1}

Suppose that \mcal{F} is a set of fields, and that
$\mcal{M}(\mcal{F})$ is the class of matroids that are
representable over every field in \mcal{F}.
It is well-known that the family of binary matroids
contains exactly two classes that arise in this way:
the binary matroids themselves, and the regular matroids.
A striking result due to Whittle~\cite{Whi97} shows that
the family of ternary matroids contains exactly six such classes
of matroids: the classes of ternary matroids, regular matroids,
near-regular matroids, dyadic matroids, sixth-roots-of-unity matroids,
and those matroids obtained from dyadic and sixth-roots-of-unity
matroids using direct sums and $2$\dash sums.

It is natural to ask for excluded-minor characterizations
of the families mentioned above.
The excluded minors for binary, ternary, and regular matroids
have been known for some time~\cite{Bix79, Sey79, Tuthom}.
Geelen, Gerards, and Kapoor~\cite{GGK} characterized
the excluded minors for $\GF(4)$\dash representable matroids.

\begin{thm}
\label{thm4}
The excluded minors for representability over $\GF(4)$ are
$U_{2,6}$, $U_{4,6}$, $P_{6}$, $F_{7}^{-}$, \nfd, $P_{8}$, and
$P_{8}''$.
\end{thm}

(Here $P_{6}$ is the rank\dash $3$ matroid with six elements,
and a triangle as its only non-spanning circuit.
Other matroids mentioned in the article are
defined in Section~\ref{excludedminors}.)
Since the class of sixth-roots-of-unity matroids is
exactly $\mcal{M}(\{\GF(3),\GF(4)\})$,
Theorem~\ref{thm4} leads to an excluded minor characterization
of the sixth-roots-of-unity matroids~\cite[Corollary~1.4]{GGK}.

In this article we consider the class of near-regular matroids,
which is exactly $\mcal{M}(\{\GF(3),\GF(4),\GF(5)\})$.
By adapting the proof of Theorem~\ref{thm4},
Geelen was able to characterize the excluded minors for
near-regularity.
However, this result remained unpublished until now.
We present a proof of Geelen's theorem.

\begin{thm}[Geelen]
\label{thm1}
The excluded minors for the class of near-regular matroids are
$U_{2,5}$, $U_{3,5}$, $F_{7}$, $F_{7}^{*}$, $F_{7}^{-}$, \nfd, \agde,
$(\agde)^{*}$, $\Delta_{T}(\agde)$, and $P_{8}$.
\end{thm}

We now give an informal outline of the proof.
The classes of regular, near-regular, sixth-roots-of-unity,
and dyadic matroids can all be characterized as the
matroids representable over a particular partial field.
Partial fields were introduced by Semple and Whittle~\cite{SW96}.
They are much like fields, except that addition is not always
defined.
If the subdeterminants of a matrix over a partial field
are all defined, then there is a corresponding matroid,
whose ground set consists of the rows and columns of the matrix.
Two matrices representing the same matroid are equivalent
if they are equal up to pivots, scaling, and applications
of partial field automorphisms.
Kahn~\cite{Ka88} showed that a stable matroid is uniquely representable
over $\GF(4)$, up to equivalence, and this fact plays a
crucial role in the proof of Theorem~\ref{thm4}.
(A stable matroid is one that cannot be expressed as a direct sum
or a $2$\dash sum of two nonbinary matroids.)

In order to proceed with our proof, we must establish a
similar uniqueness of representations for near-regular matroids.
For this purpose we use Whittle's tool of stabilizers~\cite{Whi96b}.
In Section~\ref{sct:unique} we prove an analogue of
Kahn's theorem by showing that a stable near-regular
matroid is uniquely representable over the near-regular partial
field.

We reduce the proof of
Theorem~\ref{thm1} to a finite case check by
proving that any excluded minor for near-regularity
has at most eight elements.
We suppose that $M$ is a counterexample to this
proposition.
Theorem~3.1 in~\cite{GGK} shows that
there are elements $u$ and $v$, such that $M \del u$,
$M \del v$, and $M \del \{u,v\}$ are all stable, and
$M \del \{u,v\}$ is connected and nonbinary.
At this point, Geelen et al.~construct the
unique $\GF(4)$\dash representable matroid $N$ such that
$M \del u = N \del u$ and $M \del v = N \del v$.
Our proof is slightly different, in that our
matroid $N$ need not be near-regular.
However, $N$ is representable over the field \qalpha,
as is every near-regular matroid.
Whittle's characterization reveals that the
counterexample $M$ cannot be \qalpha\dash representable,
so $M$ and $N$ are genuinely different.

The core of the proof is contained in
Section~\ref{sec:reduction}.
This part of the proof follows the proof of
Theorem~\ref{thm4} very closely, only deviating when that
proof calls upon the structure of $\GF(4)$.
We are advantaged here by the fact that our counterexample
must be ternary.
In the proof of Theorem~\ref{thm4}, there is no
a priori reason why the counterexample need be
representable over any field.
Our fundamental tool is the uniqueness of the
matroid $N$.
Suppose that $M'$ is some small proper (and hence
near-regular) minor of $M$, such that
$M'\del u$, $M'\del v$, and $M'\del \{u,v\}$ are all
stable, and $M'\del \{u,v\}$ is connected and nonbinary.
By using the same technique as before, we can construct a
\qalpha\dash representable matroid $N'$ such that
$M'\del u = N'\del u$ and $M'\del u = N'\del v$.
The uniqueness of $N$ guarantees that $N'$ is the minor of
$N$ that corresponds to $M'$, and that $M'=N'$.
If we can find some certificate that $M'$ and $N'$ are not equal,
then we have arrived at a contradiction.
This contradiction forces us to conclude that $M'$ is
not near-regular, and that therefore $M'=M$.
Because we have a bound on the size of $M'$, this
induces a bound on the size of $M$.

In order to invoke the uniqueness of $N$,
certain connectivity conditions have to be satisfied.
To obtain these conditions we use blocking
sequences, which we review in Section~\ref{sec:machinery}.

Once we have completed the work of
Section~\ref{sec:reduction}, finishing the proof
is relatively straightforward.
In Section~\ref{sec:conclusion} we first introduce
the matroids listed in Theorem~\ref{thm1}, and
we show that they are in fact excluded minors for the
class of near-regular matroids.
Then it remains only to perform the finite case-check.
All undefined matroid terms are as in
Oxley~\cite{oxley}.

\section{Preliminaries}\label{prelim}

\subsection{Partial fields}\label{sec:partialfields}
The classes of regular, near-regular, dyadic, and
sixth-roots-of-unity matroids have a common characteristic:
for every such class, there is a field $\field$, and a
subgroup $\group$ of $\field^{*}$, such that a matroid
belongs to the class if and only if it can be represented
by a matrix $A$ over $\field$, where all the
nonzero subdeterminants of $A$ belong to $\group$.
Partial fields provide a unified framework for studying this
phenomenon.
They were introduced by Semple and Whittle~\cite{SW96},
and studied further by Pendavingh and Van Zwam~\cite{PZ08conf,PZ08lift}.

Semple and Whittle developed partial fields axiomatically.
We treat them somewhat differently: Vertigan showed that every
partial field can be thought of as a ring along with a subgroup
of units (see~\cite[ Theorem~2.16]{PZ08lift}),
and we use this description as our definition.

\begin{dfn}
A \emph{partial field} is a pair $(\ring, \group)$, where $\ring$
is a commutative ring with identity, and $\group$ is a subgroup of the
group of units $\ring^*$ of $\ring$, such that $-1 \in \group$.
\end{dfn}

Note that every field $\field$ is also a partial field
$(\field, \field^{*})$.
Suppose that $\parf= (\ring,\group)$ is a partial field.
We also use \parf\ to denote the set $\group \cup 0$, so
we say that $p \in \ring$ is an element of $\parf$ (and we
write $p \in \parf$), if $p \in \group$ or $p = 0$.
Thus, $p + q$ may not be an element of $\parf$, even though
both $p$ and $q$ are contained in $\parf$.
If $p+q$ is in \parf, then we say that $p+q$ is \emph{defined}.
We use $\parf^{*}$ to denote the set of nonzero elements of
$\parf$; thus $\parf^{*} = \group$.

\begin{dfn}
\label{dfn6}
Suppose that $\parf$ is a partial field.
We say that $p \in \parf$ is a \emph{fundamental element}
if $1 - p \in \parf$.
\end{dfn}

Note that $p+q$ is defined if and only if $-q/p$ is a
fundamental element, since $p+q = p ( 1 - (-q/p))$.

\begin{dfn}
Suppose that $\parf_{1}$ and $\parf_{2}$
are partial fields.
A function $\psi: \parf_1 \rightarrow \parf_2$ is a \emph{partial-field homomorphism} if
  \begin{enumerate}
    \item $\psi(1) = 1$;
    \item for all $p, q \in \parf_1$, $\psi(pq) = \psi(p)\psi(q)$; and
    \item for all $p, q \in \parf_1$ such that $p+q$ is defined, $\psi(p) + \psi(q) = \psi(p+q)$.
  \end{enumerate}
\end{dfn}
In particular, if $\parf_1 = (\ring_1,\group_1)$, $\parf_2 = (\ring_2,\group_2)$, and $\psi:\ring_1\rightarrow\ring_2$ is a ring homomorphism such that $\psi(\group_1)\subseteq\group_2$, then the restriction of $\psi$ to $\parf_{1}$ is a partial-field homomorphism. It is easy to verify that if $\psi$ is a partial-field homomorphism then $\psi(0) = 0$ and $\psi(-1) = -1$.

A partial field \emph{isomorphism} from $\parf_{1}$ to
$\parf_{2}$ is a bijective homomorphism $\psi$ with the
additional property that $\psi(p) + \psi(q)$ is defined
if and only if $p + q$ is defined.
We use $\parf_{1} \cong \parf_{2}$ to denote the fact that
$\parf_{1}$ and $\parf_{2}$ are isomorphic.
An \emph{automorphism} of a partial field $\parf$ is an
isomorphism from $\parf$ to itself.

\subsection{Representation matrices}
Suppose that $A$ is a matrix with entries from a partial field
$\parf$, and that the rows and columns of $A$ are labeled by
the (ordered) sets $X$ and $Y$ respectively, where $X \cap Y = \emptyset$.
If the determinant of every square submatrix of $A$
is contained in $\parf$, then we say that $A$ is
an \emph{$X \times Y$ $\parf$\dash matrix}.
If $A$ is a $\parf$\dash matrix, then the \emph{rank} of
$A$, written $\matrank(A)$, is the largest value $k$ such that
$A$ contains a nonzero $k \times k$ subdeterminant.

Since we will frequently work with submatrices, it is useful to introduce some notation. If $X' \subseteq X$ and $Y' \subseteq Y$, then $A[X', Y']$ is the submatrix of $A$ induced by $X'$ and $Y'$. In particular, we define $A_{xy}$
to be $A[\{x\},\{y\}]$.
If $Z\subseteq X\cup Y$, then $A[Z] := A[Z\cap X, Z\cap Y]$, and $A-Z := A[X\setminus Z, Y\setminus Z]$. If $A$ is a matrix over the partial field $\parf$, and
$\psi$ is a function on $\parf$, then $\psi(A)$ is obtained
by operating on each entry in $A$ with $\psi$.
The following theorem follows from~\cite[Theorem~3.6]{SW96}
(see also~\cite[Theorem~2.8]{PZ08conf}).

\begin{lem}
\label{thm3}
Let $\parf$ be a partial field, and let $A$ be an
$X \times Y$ $\parf$\dash matrix.
Let
\begin{displaymath}
\mcal{B} := \{X\} \cup \big\{ X \symdiff Z \,\,\big|\,\,
|X\cap Z| = |Y\cap Z|,\, \det(A[Z]) \neq 0\big\}.
\end{displaymath}
Then $\mcal{B}$ is the set of bases of a matroid on $X \cup Y$.
\end{lem}

Let $M$ be the matroid of Lemma~\ref{thm3}.
We say that $M$ is \emph{representable} over $\parf$,
or is \emph{$\parf$\dash representable}, and we say that
$M$ is \emph{represented} by $A$.
We remark that this terminology is not standard:
the usual convention is that a matroid represented
by a matrix $A$ has the set of columns of $A$
as its ground set.
Throughout this article, when we say that $M$ is
represented by $A$, we mean that $M$ is the matroid of
Lemma~\ref{thm3}, so the ground set of $M$ is
the set of rows and columns of $A$, and the set of rows of $A$
is a basis of $M$.
If $M$ is represented by $A$ (in our sense), then it is
represented (in the standard sense) by the matrix
obtained from $A$ by appending an $|X|\times |X|$
identity matrix.
For this reason, we write $M = M[I|A]$ if
$A$ is a $\parf$\dash matrix, and $M$ is the
matroid in Lemma~\ref{thm3}.
If $A$ is an $X\times Y$ $\parf$\dash matrix, and
$M$ is the matroid represented by $A$, then $M^{*}$ is
represented by $A^{T}$, the transpose of $A$, where the rows
and columns of $A^{T}$ are labeled with $Y$ and $X$ respectively.

\begin{prop}
\label{prop9}
\tup{\cite[Proposition~4.2]{SW96}}.
Let $\parf$ be a partial field.
The class of $\parf$\dash representable matroids is closed
under duality, taking minors, direct sums, and $2$\dash sums.
\end{prop}

The next result follows from~\cite[Proposition~5.1]{SW96}
or~\cite[Proposition~2.10]{PZ08lift}.

\begin{prop}
\label{prop:hom}Let $\parf_1, \parf_2$ be partial fields and let $\psi: \parf_1 \rightarrow \parf_2$ be a homomorphism.
Let $A$ be a $\parf_1$\dash matrix. Then
\begin{enumerate}
\item $\psi(A)$ is a $\parf_2$\dash matrix;
\item If $A$ is square then $\det(A) = 0$ if and only if
$\det(\psi(A)) = 0$; and
\item $M[I|A] = M[I|\psi(A)]$.
\end{enumerate}
\end{prop}

\begin{dfn}\label{def:pivot}Let $A$ be an $X\times Y$ $\parf$\dash matrix, and let $x \in X$, $y \in Y$ be such that $A_{xy} \neq 0$. Then we define $A^{xy}$ to be the $(X \symdiff \{x,y\}) \times (Y\symdiff \{x,y\})$ matrix given by
\begin{displaymath}
  (A^{xy})_{uv} =
\begin{cases}
    A_{xy}^{-1} \quad & \textrm{if } uv = yx\\
    A_{xy}^{-1} A_{xv} & \textrm{if } u = y, v\neq x\\
    -A_{xy}^{-1} A_{uy} & \textrm{if } v = x, u \neq y\\
    A_{uv} - A_{xy}^{-1} A_{uy} A_{xv} & \textrm{otherwise.}
\end{cases}
\end{displaymath}
\end{dfn}
We say that $A^{xy}$ is obtained from $A$ by \emph{pivoting} over $xy$.
Note that after pivoting, $x$ labels a column, and $y$ labels
a row.
Suppose that $\parf$ is a partial field and that $A$
is an $X \times Y$ $\parf$\dash matrix.
\emph{Scaling} means multiplying the rows or columns
of $A$ by nonzero members of $\parf$.
The next result is
Proposition~3.3 in~\cite{SW96}, or
Proposition~2.5 in~\cite{PZ08lift}.

\begin{prop}
\label{prop11}
If $A$ is a $\parf$\dash matrix, and $A'$ is obtained from $A$ by
scaling and pivoting, then $A'$ is a $\parf$\dash matrix.
\end{prop}

\begin{dfn}
\label{dfn3}
Let $\parf$ be a partial field, and let $A, A'$ be
$\parf$\dash matrices.
Then $A$ and $A'$ are \emph{scaling-equivalent} if $A'$ can
be obtained from $A$ by scaling.
If $A'$ can be obtained from $A$ by scaling, pivoting,
permuting columns and rows (permuting labels at the same time),
and applying automorphisms of $\parf$, then we say that $A$
and $A'$ are \emph{equivalent}.
\end{dfn}

The next result follows easily from~\cite[Proposition~3.5]{SW96}
and Proposition~\ref{prop:hom}.

\begin{prop}
\label{prop13}
Suppose that $A$ and $A'$ are equivalent $\parf$\dash matrices.
Then $M[I|A] = M[I|A']$.
\end{prop}

\begin{dfn}
  Let $M$ be a matroid and suppose that $\parf$ is a partial field. We say that $M$ is \emph{uniquely representable over $\parf$} if, whenever $A, A'$ are $\parf$\dash matrices such that $M = M[I|A] = M[I|A']$, then $A$ and $A'$ are
equivalent.
\end{dfn}

\subsection{Bipartite graphs and twirls}
Let $M$ be a rank\dash $r$ matroid with ground set $E$, and let
\mcal{B} be its set of bases.
Suppose that $B \in \mcal{B}$.
Let $G_B(M) = (V,E)$ be the bipartite graph with vertices $V := B \cup (E\setminus B)$ and edges $E := \{(x,y) \in B\times (E\setminus B) \mid B\symdiff \{x,y\} \in \mathcal{B}\}$.

Let $A$ be an $X\times Y$ matrix.
We associate with $A$ a bipartite graph $\bip(A) = (V,E)$, where
$V: = X \cup Y$ and $E := \{(x,y) \in X\times Y \mid A_{xy} \neq 0\}$.
Thus each edge, $e$, of $\bip(A)$ corresponds to a nonzero entry,
$A_{e}$, of $A$.
We note here that if $A_{xy} \neq 0$, and $y'$ and $x'$ are
neighbors of $x$ and $y$ respectively such that
$y'$ and $x'$ are not adjacent in $\bip(A)$, then
$y'$ and $x'$ are adjacent in $\bip(A^{xy})$.

\begin{lem}\label{lem:bipproperties}
Let $\parf$ be a partial field, $A$ an $X\times Y$ $\parf$\dash matrix,
and let $M=M[I|A]$.
\begin{enumerate}
\item \label{eq:bipnonzeroentries}$G_X(M) = \bip(A)$.
\item \label{eq:bipscaling}Let $T$ be a forest of $\bip(A)$ with edges $e_1, \ldots, e_k$. Suppose that $p_1, \ldots, p_k$ are elements of $\parf^*$.
There exists a $\parf$\dash matrix $A'$ such that $A'$ is
scaling-equivalent to $A$, and $A_{e_{i}}' = p_{i}$ for $1 \leq i \leq k$.
  \end{enumerate}
\end{lem}

\begin{proof}
Suppose that $x \in X$ and $y \in Y$.
Then $xy$ is an edge of $G(A)$ if and only if the determinant
of $A[\{x\},\{y\}]$ is nonzero, which is true if and only
if $X \symdiff \{x,y\}$ is a basis of $M$.
This is equivalent to $xy$ being an edge of $G_{X}(M)$.

We prove the second statement by induction on $k$.
The result is trivially true if $T$ contains no edges.
By relabeling as required, we can assume that in the forest
$T$, the edge $e_{k}$ is incident with a degree-one vertex $v$.
By induction, there is a matrix $A''$ obtained from $A$ by
scaling, with the property that $A_{e_{i}}'' = p_{i}$ for
$1 \leq i \leq k-1$.
Certainly $A_{e_{k}}''$ is nonzero, let us say that
it is equal to the element $\beta \in \parf^{*}$.
Now we multiply the row or column labeled by $v$ with
$p_{k}\beta^{-1}$ to produce $A'$.
\end{proof}

Let $A$ be a matrix and suppose that $T$ is a forest of $\bip(A)$.
We say that $A$ is \emph{$T$\dash normalized} if $A_{xy} = 1$ for all $xy \in T$. By Lemma~\ref{lem:bipproperties} there is always a $T$\dash normalized
matrix $A'$ that is scaling-equivalent to $A$.

We make repeated use of the following (easy) fact.

\begin{prop}
\label{prop2}
Let $G$ be a graph, and suppose that $S$ is a set of edges that
contains a maximal forest of $G$.
If $e$ is an edge not contained in $S$, then there is an induced
cycle of $G$ that contains $e$, and such that the edges of this
cycle are contained in $S \cup e$.
\end{prop}

\begin{dfn}\label{def:twirl}
  Let $A$ be a square $\parf$\dash matrix. Then $A$ is a \emph{twirl} if $\bip(A)$ is a cycle and $\det(A) \neq 0$.
\end{dfn}

Recall that the rank\dash $r$ \emph{whirl} is denoted by $\mcal{W}^{r}$.
A whirl is representable over a field $\field$ if and only
if $|\field| \geq 3$.
Note that if $A$ is a twirl then $M[I|A]$ is a whirl.

\begin{prop}
\tup{\cite[Proposition 4.5]{GGK}}.
\label{GGK4.5}
  Let $A$ be an $X\times Y$ matrix that is a twirl, and let $x$, $y$
be such that $A_{xy} \neq 0$.
  \begin{enumerate}
    \item If $|X\cup Y| = 4$ then $A^{xy}$ is a twirl.
    \item If $|X\cup Y| > 4$ then $A^{xy}-\{x,y\}$ is a twirl.
  \end{enumerate}
\end{prop}

\subsection{Near-regular matroids}
\label{sct2.2}
Recall that \qalpha\ is the field obtained
from the rational numbers by extending with the
transcendental element $\alpha$.
Let $\Z[\alpha, 1/\alpha, 1/(1-\alpha)]$
be the subring of \qalpha\ induced by $\alpha$, $1/\alpha$,
$1/(1-\alpha)$, and the integers.

\begin{dfn}
  The \emph{near-regular} partial field is
  \begin{displaymath}
    \nreg := \left(\Z\left[\alpha, \frac{1}{\alpha}, \frac{1}{1-\alpha}\right], \left\langle -1, \alpha, 1-\alpha \right\rangle \right).
  \end{displaymath}
\end{dfn}

Here $\langle -1, \alpha, 1-\alpha \rangle$ denotes
the subgroup of units generated by $-1$, $\alpha$, and
$1-\alpha$.
Thus $\nreg$ consists of zero, and elements of
the form $\pm \alpha^{i}(1-\alpha)^{j}$, where $i$ and
$j$ are integers.
We note that $\nreg$ is a special case of a class of partial
fields studied by Semple~\cite{Sem97}.

A $\nreg$\dash matrix is said to be \emph{near-unimodular}.
A matroid is \emph{near-regular} if it is representable over
$\nreg$.
Whittle's characterization shows, amongst other things, that
a matroid is near-regular if and only if it is representable over
every field with cardinality at least three.

\begin{thm}\label{thm:nearregreps}
\tup{\cite[Theorem~1.4]{Whi97}}.
  Let $M$ be a matroid. The following are equivalent:
  \begin{enumerate}
    \item $M$ is representable over $\GF(3)$, $\GF(4)$, and $\GF(5)$;
    \item $M$ is representable over $\GF(3)$ and $\GF(8)$;
    \item $M$ is near-regular; and
    \item $M$ is representable over all fields except, possibly, $\GF(2)$.
  \end{enumerate}
\end{thm}

Next we collect some basic facts about the near-regular partial field.
The first result follows from Lemmas~2.23 and~4.3 in~\cite{PZ08lift}.
\begin{prop}
\label{prop15}
The fundamental elements of $\nreg$ are
\begin{displaymath}
\left\{0, 1, \alpha, 1 - \alpha, \frac{1}{1 - \alpha},
\frac{\alpha}{\alpha - 1}, \frac{\alpha-1}{\alpha},
\frac{1}{\alpha}\right\}.
\end{displaymath}
\end{prop}

\begin{prop}
\label{prop16}
Let $\alpha_{i}$ and $\alpha_{j}$ be fundamental elements of $\nreg$
that are equal to neither $1$ nor $0$.
There is an automorphism of $\nreg$ that takes $\alpha_{i}$ to $\alpha_{j}$.
\end{prop}

\begin{proof}
Obviously an automorphism of $\nreg$ permutes the
fundamental elements.
Consider a function $\psi:\qalpha \to \qalpha$ which acts as the
identity on $0$ and $1$, takes $\alpha$ to another fundamental
element of $\nreg$, and which respects addition and
multiplication.
The following table shows how $\psi$
acts upon the element $\alpha^{i}(1 - \alpha)^{j}$ of $\nreg$.
\begin{displaymath}
\begin{array}{ccc}\hline
\psi(\alpha) & \quad & \psi(\alpha^{i}(1-\alpha)^{j})\\[0.5ex]\hline
\alpha &&\alpha^{i}(1-\alpha)^{j}\\[0.5ex]
1-\alpha&&\alpha^{j}(1-\alpha)^{i}\\[0.5ex]
1/(1-\alpha)&&(-1)^{j}\alpha^{j}(1-\alpha)^{-(i+j)}\\[0.5ex]
\alpha/(\alpha-1)&&(-1)^{i}\alpha^{i}(1-\alpha)^{-(i+j)}\\[0.5ex]
(\alpha-1)/\alpha&&(-1)^{i}\alpha^{-(i+j)}(1-\alpha)^{i}\\[0.5ex]
1/\alpha&&(-1)^{j}\alpha^{-(i+j)}(1-\alpha)^{j}\\\hline
\end{array}
\end{displaymath}
Now it is clear that the restriction of $\psi$ to $\nreg$
is indeed an automorphism.
Since the inverse of an automorphism is another automorphism,
and so is the composition of two automorphisms, 
the result follows easily.
\end{proof}

Recall that a matrix over the rationals is \emph{totally unimodular}
if every subdeterminant belongs to $\{0,1,-1\}$.
A matroid is \emph{regular} if and only if it can be represented
by a totally unimodular matrix.
It is well-known that regular matroids are representable over all fields
(\cite[Theorem~6.6.3]{oxley}).

\begin{prop}
\label{autofixed}
Suppose that $A$ is a near-unimodular matrix that is not
equivalent to a totally unimodular matrix.
If $\psi$ is an automorphism of $\nreg$ such that
$\psi(A) = A$, then $\psi$ is the trivial automorphism.
\end{prop}

\begin{proof}
Suppose that the rows and columns of $A$ are labeled
with $X$ and $Y$.
We assume that $\psi$ is not the identity function on $\nreg$, so
that $\psi(\alpha) \ne \alpha$.
Let $T$ be a maximal forest of $\bip(A)$.
By examining the proof of Lemma~\ref{lem:bipproperties}, we
see that while $T$\dash normalizing $A$, we only ever multiply a
row or column by the inverse of a nonzero entry of $A$.
If $\beta$ is a nonzero entry of $A$, then $\psi(\beta)=\beta$,
and therefore $\psi(\beta^{-1})=\beta^{-1}$.
It follows easily that normalizing $A$ does not affect the
assumption that $\psi(A)=A$.
Moreover, normalizing $A$ does not produce a totally unimodular
matrix, as $A$ is not equivalent to such a matrix.
Henceforth we assume that $A$ is $T$\dash normalized.

Let $S$ be the set of nonzero entries of $A$ that are equal to
$1$ or $-1$.
There is an edge $e$ in $\bip(A)$ not contained in $S$.
As $S$ contains the edge-set of $T$, Proposition~\ref{prop2}
asserts that there is a set $C \subseteq X \cup Y$
such that $G(A[C])$ is an induced cycle of $G(A)$
containing $e$, and the edges of $G(A[C])$ are contained
in $S \cup e$.

Suppose that $A_{e} = (-1)^{k}\alpha^{i}(1-\alpha)^{j}$ for
integers $i$, $j$, and $k$.
Then
\begin{equation}
\label{eqn2}
\psi(\alpha^{i}(1-\alpha)^{j}) = \alpha^{i}(1-\alpha)^{j}.
\end{equation}
By examining the table in the proof of
Proposition~\ref{prop16}, we see that if
$\psi(\alpha)$ is equal to $1/(1-\alpha)$ or
$(\alpha-1)/\alpha$, then the only solution to
Equation~\eqref{eqn2} is $i=j=0$.
This is a contradiction as $e \notin S$.
Therefore we suppose that $\psi(\alpha) = 1-\alpha$.
Then $\psi(\alpha^{i}(1-\alpha)^{j}) = \alpha^{j}(1-\alpha)^{i}$,
so $i=j$.

Since every nonzero entry in $A[C]$, other than
$A_{e}$, is in $\{1,-1\}$, and $\bip(A[C])$ is a cycle, it
follows that the determinant of $A[C]$ is,
up to multiplication by $-1$, equal to
$A_{e} \pm 1$.
As this determinant belongs to $\nreg$, it follows that
either $A_{e}$ or $-A_{e}$ is a fundamental element.
But no fundamental element, other than $1$, is of the form
$\pm\alpha^{i}(1-\alpha)^{i}$, and we have a contradiction.

Similarly, if $\psi(\alpha)$ is equal to
$\alpha/(\alpha-1)$ or $1/\alpha$, then $i$ and $j$ must
satisfy either $2j = -i$, or $2i = -j$.
In either case we arrive at a similar contradiction.
\end{proof}

The next result is an adaptation of Lemma~4.3 in~\cite{GGK}.

\begin{lem}\label{GGK4.3}
 Let $A$ be a near-unimodular $X\times Y$ matrix. Then there is some $C\subseteq X\cup Y$ such that $A[C]$ is a twirl if and only if $M[I|A]$ is nonbinary.
\end{lem}
\begin{proof}
If $A$ contains a twirl, then $M[I|A]$ contains a whirl-minor,
and is therefore nonbinary.
For the converse, let $T$ be a maximal forest of $\bip(A)$, and
assume that $A$ is $T$\dash normalized.
Let $S$ be the set of nonzero entries in $A$ that are equal to
$1$ or $-1$.
As $M[I|A]$ is nonbinary, it is certainly not regular, and therefore
$A$ is not totally unimodular.
Hence there is an edge $e$ of $\bip(A)$ such that
$e \notin S$.
Proposition~\ref{prop2} provides a subset $C \subseteq X \cup Y$
such that $G(A[C])$ is a cycle containing $e$,
and the edges of $G(A[C])$ are contained in $S\cup e$.
Then $A[C]$ is a twirl.
\end{proof}

The following analogue of Lemma~4.4 in~\cite{GGK}
is proved in a similar way to Lemma~\ref{GGK4.3}.

\begin{lem}
\label{GGK4.4}
Let $A$ be an $X\times Y$ $\nreg$\dash matrix, and suppose
that $A[C]$ is a twirl for some $C\subseteq X\cup Y$.
Let $v_{0},\ldots, v_{p}$ be the vertices of $A[C]$ in cyclic
order.
Suppose that $x \in (X\cup Y) \setminus C$ and the neighbors
of $x$ in $C$ are $v_{i_{1}},\ldots, v_{i_{k}}$, where $k \geq 2$.
Then there exists a twirl of the form $A[\{x,v_{i_{j}},\ldots, v_{i_{j+1}}\}]$
(where $1 \leq j \leq k-1$) or
$A[\{x,v_{0},\ldots, v_{i_{1}},v_{i_{k}},\ldots,v_{p}\}]$.
\end{lem}

\subsection{Stabilizers}
The notion of a stabilizer, introduced by Whittle~\cite{Whi96b}, is an indispensable tool for controlling inequivalent representations.
\begin{dfn}
Let $\parf$ be a partial field, and let $M$ and $N$ be
$3$\dash connected \parf\dash representable matroids such that
$N$ is a minor of $M$.
Suppose that the ground set of $N$ is $X' \cup Y'$, where
$X'$ is a basis of $N$.
We say that $N$ is a \emph{\parf\dash stabilizer} for $M$ if,
whenever $A_{1}$ and $A_{2}$ are $X \times Y$ \parf\dash matrices
(where $X' \subseteq X$ and $Y' \subseteq Y$) such that
\begin{enumerate}
\item $M = M[I|A_{1}] = M[I|A_{2}]$;
\item $A_{1}[X',Y']$ is scaling-equivalent to $A_{2}'[X',Y']$; and
\item $N = M[I|A_{1}[X',Y']] = M[I|A_{2}[X',Y']]$,
\end{enumerate}
then $A_{1}$ is scaling-equivalent to $A_{2}$.
\end{dfn}
We say that $N$ is a $\parf$\dash stabilizer for a class of matroids
if $N$ is a $\parf$\dash stabilizer for every $3$\dash connected member
of the class.

Whittle proved that verifying that a matroid is a stabilizer can
be accomplished with a finite case-check.
(See also~\cite[Theorem~3.10]{PZ08conf}.)
\begin{thm}[Stabilizer Theorem, Whittle~\cite{Whi96b}]\label{cor:stabilizer}
  Let $\parf$ be a partial field, and let $N$ be a $3$\dash connected $\parf$\dash representable matroid. Let $M$ be a $3$\dash connected $\parf$\dash representable matroid having an $N$\dash minor. Then exactly one of the following is true:
  \begin{enumerate}
    \item \label{eq:stablemat2} $N$ stabilizes $M$.
    \item \label{eq:notstablemat2} $M$ has a $3$\dash connected minor $M'$ such that
    \begin{enumerate}
      \item $N$ does not stabilize $M'$;
      \item $N$ is isomorphic to $M'\con x$, $M'\del y$, or $M'\con x \del y$, for some $x,y \in E(M')$; and
      \item If $N$ is isomorphic to $M'\con x \del y$ then at least one of $M'\con x, M'\del y$ is $3$\dash connected.
    \end{enumerate}
  \end{enumerate}
\end{thm}
Since $U_{2,4}$ has no $3$\dash connected, near-regular one-element
extensions or coextensions, the following result follows easily:
\begin{cor}\label{cor:stab}
  $U_{2,4}$ is a $\nreg$\dash stabilizer for the class of near-regular matroids.
\end{cor}

\subsection{The $\Delta\textrm{-}Y$ operation}
\label{sct2.3}
Suppose that $M$ is a matroid and that $T$ is a
coindependent triangle of $M$.
Let $N$ be an isomorphic copy of $M(K_{4})$ such that
$E(N) \cap E(M) = T$ and $T$ is a triangle of $N$.
Then the \emph{generalized parallel connection} of $M$
and $N$, denoted $P_{T}(N, M)$, is defined.
This is the matroid on the ground set $E(M) \cup E(N)$
whose flats are exactly the sets $F$ such that
$F \cap E(N)$ and $F \cap E(M)$ are flats of $N$ and $M$
respectively.
Suppose that $T = \{a, b, c\}$.
If $x \in T$, then there is a unique element, $x'$, of $N$,
that is in no triangle with $x$.
We swap the labels on $x$ and $x'$ in $P_{T}(M,N)$, for
each $x \in T$.
Thus $P_{T}(M,N) \del T$ and $M$ have the same ground set.
We say that $P_{T}(M,N) \del T$ is produced by a
\emph{\dy\ operation} on $M$, and we denote the
resulting matroid with $\Delta_{T}(M)$.
The \dy\ operation has been studied by Akkari and Oxley~\cite{AO93}
and generalized by Oxley, Semple, and Vertigan~\cite{OSV00}.

Suppose that $T$ is an independent triad of the matroid $M$.
Then $\Delta_{T}(M^{*})$ is defined, and
$(\Delta_{T}(M^{*}))^{*}$ is said to be produced from
$M$ by a \emph{\yd\ operation}, and is denoted
by $\nabla_{T}(M)$.
The next results follow by combining Lemmas~2.6 and~2.11,
and Theorem~1.1 in~\cite{OSV00}.

\begin{lem}
\label{DYrank}
Suppose that $T$ is a coindependent triangle of $M$.
Then
\begin{displaymath}
\rank(\Delta_{T}(M)) = \rank(M)+1.
\end{displaymath}
Moreover, $T$ is an independent triad in $\Delta_{T}(M)$,
and $\nabla_{T}(\Delta_{T}(M)) = M$.
\end{lem}

\begin{lem}
\label{lem4}
Suppose that $\parf$ is a partial field and that $M$
is an excluded minor for the class of $\parf$\dash representable
matroids.
If $T$ is a coindependent triangle of $M$ then
$\Delta_{T}(M)$ is also an excluded minor for the class
of $\parf$\dash representable matroids.
\end{lem}

\section{Unique representations}
\label{sct:unique}

In this section we prove an analogue of Kahn's theorem
by showing that stable near-regular matroids are
uniquely representable over $\nreg$.
Brylawski and Lucas~\cite{BL76} prove that binary matroids are
uniquely representable over any field.
The proof of the following result sketches the straightforward
adaptation of their argument to partial fields.

\begin{prop}
\label{binuniq}
Suppose that $\parf$ is a partial field, and that
the $X \times Y$ $\parf$\dash matrices $A_{1}$ and $A_{2}$
both represent the binary matroid $M$.
Let $T$ be a maximal forest of $\bip(A_{1}) = \bip(A_{2})$.
Suppose that both $A_{1}$ and $A_{2}$ are $T$\dash normalized.
Then $A_{1} = A_{2}$.
Hence $M$ is uniquely representable over $\parf$.
\end{prop}

\begin{proof}
We claim that $A_{1} = A_{2}$ and that every nonzero
entry of $A_{1}$ and $A_{2}$ belongs to $\{1,-1\}$.
Let $S$ be the set of edges of $\bip(A_{1}) = \bip(A_{2})$
such that $xy \in S$ if and only if
$(A_{1})_{xy} \in \{1,-1\}$ and $(A_{2})_{xy} = (A_{1})_{xy}$.
If our claim is false, then there is an edge $e$ of
$\bip(A_{1})$ not in $S$.
Since $S$ contains the edge-set of $T$, Proposition~\ref{prop2}
implies that there is a set $C \subseteq X \cup Y$
such that $G(A_{1}[C])$ is a cycle containing $e$, the edges
of which are contained in $S \cup e$.

Let $A$ be the $X \times Y$ $\GF(2)$\dash matrix obtained
from $A_{1}$ by replacing every nonzero entry with $1$.
As $M$ is binary, $A$ represents $M$ over $\GF(2)$.
Since $G(A[C])$ is a cycle,
it is easy to see that $A[C]$ has zero determinant over $\GF(2)$.
Therefore the determinant of $A_{1}[C]$ is also zero.
Let $\beta = (A_{1})_{e}$.
Every nonzero entry of $A_{1}[C]$, other than $(A_{1})_{e}$,
belongs to $\{1,-1\}$.
Now it is easy to see that the determinant of
$A_{1}[C]$ is, up to multiplication by $-1$, equal to
$\beta \pm 1$.
Thus $\beta \in \{1,-1\}$.
However, the same argument shows that $(A_{2})_{e}$
is equal to $\beta$, and we have a contradiction to the
fact that $e \notin S$.
\end{proof}

The direct sum or $2$\dash sum of two uniquely representable
matroids need not be uniquely representable (for example, the
$2$\dash sum of two copies of $U_{2,4}$ is not uniquely representable
over $\GF(4)$).
But we do have the following partial result.

\begin{prop}
\label{uniqsum}
Let $\parf$ be a partial field, and suppose that
the matroid $M_{1}$ is uniquely representable over $\parf$.
Let $M_{2}$ be a $\parf$\dash representable matroid,
and suppose that, whenever $A_{1}$ and $A_{2}$ are two
$T$\dash normalized $X \times Y$ $\parf$\dash representations
of $M_{2}$, then $A_{1} = A_{2}$.
(Here $T$ is a maximal forest of $\bip(A_{1}) = \bip(A_{2})$.)
Then $M_{1} \oplus_{1} M_{2}$ and $M_{1} \oplus_{2} M_{2}$
are uniquely $\parf$\dash representable.
\end{prop}

\begin{proof}
We present the proof that $M_{1} \oplus_{2} M_{2}$ is
uniquely representable.
The proof for $M_{1} \oplus_{1} M_{2}$ is similar
(and easier).

Let $A_{1}$ and $A_{2}$ be two $\parf$\dash representations of
$M_{1} \oplus_{2} M_{2}$.
Let $X$ be a basis of $M_{1} \oplus_{2} M_{2}$, and let
$Y = E(M_{1} \oplus_{2} M_{2}) - X$.
By pivoting, we can assume that $A_{1}$ and $A_{2}$
are $X \times Y$ matrices.
Thus $(A_{1})_{xy}$ is nonzero if and only if
$(A_{2})_{xy}$ is nonzero.
For $i = 1,2$, let $X_{i}$ and $Y_{i}$ be equal to
$X \cap E(M_{i})$ and $Y = Y \cap E(M_{i})$ respectively.
It is straightforward to prove (see Lemma~\ref{lem:matrixconn})
that, by relabeling as necessary, we can assume that
$A_{i}[X_{2},Y_{1}]$ is the zero matrix, and $A_{i}[X_{1},Y_{2}]$ has
rank one.
Therefore the nonzero columns of $A_{i}[X_{1},Y_{2}]$ are equal,
up to scaling; the same comment applies to the rows.

Let $y \in Y_{2}$ be such that $A_{i}[X_{1},\{y\}]$ is nonzero for $i=1,2$.
(Note that such a $y$ exists, for otherwise we can reduce
to the direct-sum case.)
By considering the result of contracting $X_{2}$, it is
easy to see that $A_{1}[X_{1},Y_{1} \cup y]$ and
$A_{2}[X_{1},Y_{1} \cup y]$ are representations of $M_{1}$.
By unique representability, we can apply scalings
and automorphisms of $\parf$ to $A_{2}$, and assume that
$A_{2}[X_{1},Y_{1} \cup y]=A_{1}[X_{1},Y_{1} \cup y]$.
Now, since $A_{2}[X_{1},\{y\}]=A_{1}[X_{1},\{y\}]$, and
the nonzero columns of $A_{i}[X_{1},Y_{2}]$ are parallel to
$A_{i}[X_{1},\{y\}]$, for $i=1,2$, we can scale columns of
$A_{2}$ so that $A_{2}[X_{1},Y]=A_{1}[X_{1},Y]$.

Let $x \in X_{1}$ be such that
$A_{i}[\{x\},Y_{2}]$ is nonzero for $i=1,2$.
By considering the result of contracting $X_{1} - x$, we
see that $A_{i}[X_{2}\cup x,Y_{2}]$ represents $M_{2}$.

\begin{claim}\label{forest}
Let $T$ be a forest of
$\bip(A_{1}) = \bip(A_{2})$, and assume that $T$ contains all
the edges incident with $x$.
By performing row and column scalings,
we can $T$\dash normalize both $A_{1}$ and $A_{2}$, without
affecting the assumption $A_{2}[X_{1},Y]=A_{1}[X_{1},Y]$.
\end{claim}

\begin{subproof}
The proof of the claim is inductive on the number of
edges in $T$.
If $T$ contains only those edges incident with $x$,
then we can $T$\dash normalize by multiplying
column $y$ by $1/(A_{1})_{xy} = 1/(A_{2})_{xy}$ in both
$A_{1}$ and $A_{2}$, for every neighbor $y$ of $x$.
This proves the base case of the argument.

Suppose that $T$ contains edges that are not incident
with $x$.
Let $u$ be a degree-one vertex in $T$ that is not
adjacent to $x$, and let $v$ be the vertex of $T$
adjacent to $u$.
By the inductive hypothesis, we can assume that
$A_{1}$ and $A_{2}$ are both
$(T-uv)$\dash normalized, and the assumption
$A_{2}[X_{1},Y]=A_{1}[X_{1},Y]$ still holds.
If $u \in X_{2}$ then we can scale row $u$ in
$A_{i}$ by $1/(A_{i})_{uv}$, for $i = 1,2$.
The resulting matrices are $T$\dash normalized, and
agree on the submatrices induced by $X_{1}$ and $Y$.
If $u \in X_{1}$ then we can multiply row $u$ in both
$A_{1}$ and $A_{2}$ by $1/(A_{1})_{uv} = 1/(A_{2})_{uv}$,
and we see that the claim holds for $T$.
A similar argument holds if $u \in Y_{1}$.
Thus we suppose that $u \in Y_{2}$.
Since $u$ is not adjacent to $x$, it follows that
$(A_{i})_{xu} = 0$ for $i = 1,2$.
Therefore $A_{i}[X_{1},\{u\}]$ is the zero column, since
the nonzero rows of $A_{i}[X_{1},Y_{2}]$ are parallel.
It follows that we can multiply column $u$ by
$1/(A_{i})_{vu}$ for $i=1,2$ without changing
$A_{i}[X_{1},Y]$.
This completes the proof of the claim.
\end{subproof}

Now we let $T'$ be a maximal forest of the subgraph
of $\bip(A_{1}) = \bip(A_{2})$ induced by
$X_{2} \cup Y_{2} \cup x$.
Assume that $T'$ contains all the edges incident with $x$.
We extend $T'$ to a maximal forest $T$, of
$\bip(A_{1}) = \bip(A_{2})$, where $T$ also contains all edges
incident with $x$.
By Claim~\ref{forest}, we can $T$\dash normalize $A_{1}$ and
$A_{2}$ without affecting the assumption that
$A_{2}[X_{1},Y]=A_{1}[X_{1},Y]$.

Since $A_{1}[X_{2} \cup x,Y_{2}]$ and
$A_{2}[X_{2} \cup x,Y_{2}]$ are $T'$\dash normalized,
the hypotheses imply that
$A_{2}[X_{2}\cup x,Y_{2}] = A_{1}[X_{2}\cup x,Y_{2}]$.
Now we see that, by pivoting, scaling rows and columns, and
possibly applying an automorphism, we have converted $A_{1}$ and
$A_{2}$ into identical matrices.
The result follows.
\end{proof}

\begin{dfn}
  Let $M$ be a matroid. Then $M$ is \emph{stable} if it can not be expressed as the direct sum or $2$\dash sum of two nonbinary matroids.
\end{dfn}

\begin{lem}\label{lem:stableunique}
  Let $M$ be a stable near-regular matroid. Then $M$ is uniquely representable over $\nreg$.
\end{lem}

\begin{proof}
Let $M$ be a stable near-regular matroid, and suppose that the
lemma holds for all smaller matroids.
We start by assuming that $M$ is $3$\dash connected.
If $M$ is binary, then the result follows
immediately from Proposition~\ref{binuniq}.
Therefore we suppose that $M$ is nonbinary, and therefore
has a $U_{2,4}$\dash minor.
Let $A_{1}$ and $A_{2}$ be $X \times Y$ $\nreg$\dash matrices
that represent $M$.
By pivoting, we can assume that there are $2$\dash element
subsets $X' \subseteq X$ and $Y' \subseteq Y$, such that
$A_{i}[X',Y']$ represents $U_{2,4}$ for $i=1,2$.
By scaling, we can assume that
\begin{displaymath}
A_{i}[X',Y']=
\begin{bmatrix}
      1 & 1\\
      p_{i} & 1
\end{bmatrix}
\end{displaymath}
for some $p_{i} \in \nreg$.
Since $\det(A_{i}[X',Y']) = 1-p_{i}$ is defined, $p_{1}$ and
$p_{2}$ are fundamental elements.
By Proposition~\ref{prop16}, we can apply an automorphism
of $\nreg$ to $A_{2}$, and assume that
$A_{2}[X',Y'] = A_{1}[X',Y']$.
Now the lemma follows immediately from
Corollary~\ref{cor:stab}.

Hence we assume that $M$ is not $3$\dash connected,
and can therefore be expressed as a direct sum or
a $2$\dash sum of $M_{1}$ and $M_{2}$.
Since $M$ is stable, we can assume that
$M_{2}$ is binary.
It is easy to see that $M_{1}$ must be stable.
Therefore $M_{1}$ is uniquely representable
over $\nreg$ by the inductive hypothesis.
The result now follows from
Propositions~\ref{binuniq} and~\ref{uniqsum}.
\end{proof}

\section{The setup}\label{sec:setup}

In this section we collect the results that underlie
our proof strategy.
An excluded minor $M$ for near-regularity with more
than eight elements has a ``companion'' matroid $N$ that
is representable over \qalpha.
Our main objective here is to develop the tools for
constructing $N$.

Note that if an excluded minor for near-regularity is
not ternary, then it is an excluded minor for the
class of ternary matroids.
Now the following lemmas follow immediately from
Reid's characterization of ternary matroids~\cite{Bix79, Sey79},
and Proposition~\ref{prop9}.
\begin{lem}
  \label{lem:ternary}Let $M$ be an excluded minor for the
class of near-regular matroids, and assume $M$ is not
isomorphic to $U_{2,5}$, $U_{3,5}$, $F_{7}$, or $F_{7}^{*}$.
Then $M$ is ternary.
\end{lem}
\begin{lem}\label{lem:3c}
  Let $M$ be an excluded minor for the class of near-regular matroids. Then $M$ is $3$\dash connected.
\end{lem}
\begin{lem}
  \label{lem:duality}Let $M$ be an excluded minor for the class of near-regular matroids. Then $M^*$ is an excluded minor for the class of near-regular matroids.
\end{lem}

\begin{dfn}
  Suppose that $M$ is a matroid, and that $u,v \in E(M)$.
We will say that $u$, $v$ is a \emph{deletion pair} if
  \begin{enumerate}
    \item $\{u, v\}$ is coindependent;
    \item Each of $M\del u, M\del v, M\del \{u, v\}$ is stable; and
    \item $M\del \{u,v\}$ is connected and nonbinary.
  \end{enumerate}
\end{dfn}

Our definition here is slightly different from that
used in~\cite{GGK}.
The next result follows from~\cite[Theorem~3.1]{GGK}.

\begin{lem}\label{lem:delpair}
Let $M$ be a $3$\dash connected nonbinary matroid such that
$\rank(M) \geq 4$ or $\corank(M) \geq 4$.
Then, for some $M' \in \{M, M^*\}$, there is a pair of elements $u,v$
such that $M' \del \{u,v\}$ is connected, and each of $M'\del u$, $M'\del v$, $M'\del \{u, v\}$ is
a $0$\dash, $1$\dash, or $2$\dash element coextension of a
$3$\dash connected nonbinary matroid.
Hence $u,v$ is a deletion pair for $M'$.
\end{lem}

Lemmas~\ref{GGK2.2} and~\ref{GGK2.3} are analogues of
Lemmas 2.2 and 2.3 in~\cite{GGK}.
Suppose that $A$ is a matrix (not necessarily a $\nreg$\dash matrix)
over the field \qalpha, and that all the entries of $A$ belong to $\nreg$.
If $\psi$ is a homomorphism from $\nreg$ to some
other partial field, then we use $\psi(A)$ to denote the
matrix obtained by applying $\psi$ to all the entries of $A$.

\begin{lem}\label{GGK2.2}
  Suppose $M$ is a matroid, and that $u,v$ is a deletion pair of $M$ such
that $M \del u$ and $M \del v$ are near-regular.
Let $X$ be a basis of $M\del\{u,v\}$, and define $Y := E(M)\setminus X$.
Then there exists an $X\times Y$ matrix $A$ over \qalpha\ such that
  \begin{enumerate}
    \item $M[I|A-u] = M\del u$;
    \item $M[I|A-v] = M\del v$; and
    \item $A-u$ and $A-v$ are near-unimodular.
  \end{enumerate}
  Moreover, $A$ is unique up to row and column scaling and
applying automorphisms of $\nreg$.
\end{lem}

\begin{proof}
Let $A_{1}$ be a near-unimodular $X \times (Y\setminus u)$
matrix representing $M\del u$.
Likewise, let $A_{2}$ be a near-unimodular
$X \times (Y\setminus v)$ matrix representing $M\del v$.
If $u$ is a loop, then it is straightforward to confirm that the
matrix obtained from $A_{1}$ by adding a zero column satisfies
the statements of the lemma.
Therefore we assume that $u$ (and $v$, by symmetry) is not a loop.
Now $A_{1}- v$ and $A_{2}- u$ are
near-unimodular matrices representing $M \del \{u,v\}$.
Since $M \del \{u,v\}$ is stable by the definition of
a deletion pair, it follows from Lemma~\ref{lem:stableunique} that
by scaling, and applying automorphisms of $\nreg$ to $A_{2}$,
we can assume that $A_{2}- u = A_{1}- v$.
Propositions~\ref{prop:hom} and~\ref{prop11} imply that
$A_{2}$ remains near-unimodular after these operations.
Let $A$ be the matrix obtained from $A_{1}$ by adding the
column $A_{2}[X,\{u\}]$.
Since $A - u = A_{1}$ and $A - v = A_{2}$ the conditions
of the lemma clearly hold.

To prove that $A$ is unique, we first assume that $A$ is
$T$\dash normalized, for some spanning tree $T$ of $\bip(A)$
that has $u$ and $v$ as degree-one vertices.
(Such a tree exists because $M\del \{u,v\}$, and hence
$\bip(A-\{u,v\})$, is connected; neither $u$ nor $v$ is a loop;
and because $u$ and $v$ are not adjacent.)
Let $A'$ be some other $X \times Y$ matrix over $\qalpha$ that
satisfies the conditions of the lemma.
Since $A - u$ and $A' - u$ both represent $M \del u$ over
$\nreg$, and $M \del u$ is stable, we can, by scaling and
applying automorphisms of $\nreg$ to $A'$, assume that
$A' - u = A - u$.
Similarly, as $A' - v$ and $A - v$ both represent the stable
matroid $M \del v$, there are nonsingular diagonal matrices
$D_{1}$ and $D_{2}$, and an automorphism $\psi$ of $\nreg$, such
that $D_{1}\psi(A' - v)D_{2} = A - v$.

Let $xy$ be an edge in $T - \{u,v\}$.
Then
\begin{multline}
\label{eqn1}
1 = (A-v)_{xy} = (D_{1})_{xx}\psi((A'-v)_{xy})(D_{2})_{yy}\\
= (D_{1})_{xx}\psi(1)(D_{2})_{yy} = (D_{1})_{xx}(D_{2})_{yy}.
\end{multline}
Let $\gamma = (D_{1})_{xx}$, so that
$(D_{2})_{yy} = 1/\gamma$.
Let $w$ be some vertex in $T - \{u,v\}$.
It is easy to prove, using Equation~\eqref{eqn1}, and
induction on the length of the path in $T - \{u,v\}$ joining
$w$ to $x$, that if $w \in X$ then $(D_{1})_{ww} = \gamma$, and if
$w \in Y$ then $(D_{2})_{ww} = 1/\gamma$.
Thus $A - v = D_{1}\psi(A' - v)D_{2}$ is obtained
from $A' - v$ by applying $\psi$, and possibly
scaling the column $u$ by a nonzero constant.
Thus $\psi(A' - \{u,v\}) = A - \{u,v\} = A' - \{u,v\}$.
Since $A' - \{u,v\}$ represents the nonbinary
matroid $M\del \{u,v\}$, it follows that $A' - \{u,v\}$
is near-unimodular but not totally unimodular.
Proposition~\ref{autofixed} implies that $\psi$ is
the trivial automorphism.
Thus $A-v$ can be obtained from $A'-v$ by possibly
scaling the column $u$.
Now, as $A'[X,v] = A[X,v]$, it follows that
$A'$ and $A$ are equal, up to scaling and automorphisms of $\nreg$.
\end{proof}

We will need a few more properties of the matrix appearing in
Lemma~\ref{GGK2.2}.
First of all, we need to be able to modify the choice of the
basis $X$.
The straightforward proof of the next result is omitted.

\begin{lem}\label{lem:pivot}
  Suppose that $M$ is a matroid, and that $u,v$ is a deletion pair of $M$
such that $M \del u$ and $M \del v$ are near-regular.
Let $X$ be a basis of $M \del \{u,v\}$, and let $Y = E(M) \setminus X$.
Let $A$ be the $X\times Y$ \qalpha\dash matrix such that
$M[I|A-u] = M\del u$, $M[I|A-v]=M \del v$, and $A-u$ and $A-v$ are
near-unimodular.
Suppose that $x \in X$, $y \in Y\setminus\{u,v\}$ and that $A_{xy} \neq 0$.
Then $M[I|A^{xy}-u]=M\del u$,
$M[I|A^{xy}-v]=M \del v$, and $A^{xy}-u$ and $A^{xy}-v$
are near-unimodular.
\end{lem}

Consider the function from $\nreg$ to $\GF(3)$ which
takes $0$ to $0$, $1$ to $1$, and $\alpha$ to $-1$.
It is not difficult to confirm that this induces a
partial-field homomorphism from $\nreg$ to $\GF(3)$.
Indeed, if $\phi:\nreg \to \GF(3)$ is a partial-field
homomorphism, then $\phi(0)=0$ and $\phi(1) = 1$,
by elementary properties of homomorphisms, and
$\phi(\alpha)$ cannot be equal to $0$, as
$\phi(\alpha)$ must have a multiplicative inverse.
Nor, for the same reason, can $\phi(1-\alpha)$ be equal to
$0$.
Thus $\phi(\alpha) = -1$, so there is a unique
partial-field homomorphism from $\nreg$ to $\GF(3)$.

\begin{lem}
\label{lem:ternaryhom}
Suppose $M$ is a ternary matroid, and that $u,v$ is a deletion
pair of $M$ such that $M \del u$ and $M \del v$ are near-regular.
Let $X$ be a basis of $M \del \{u,v\}$, and let $Y = E(M) \setminus X$.
Let $A$ be the $X\times Y$ \qalpha\dash matrix such that
$M[I|A-u] = M\del u$, $M[I|A-v]=M \del v$, and $A-u$ and $A-v$ are
near-unimodular.
Let $\phi$ be the homomorphism from $\nreg$ to $\GF(3)$.
Then $M = M[I|\phi(A)]$.
\end{lem}

\begin{proof}
We assume that $\phi(A)$ is $T$\dash normalized for
some maximal forest $T$ of $\bip(\phi(A))$, where $u$ and
$v$ are degree-one vertices of $T$.
Let $A'$ be an $X \times Y$ $\GF(3)$\dash matrix that
represents $M$.
Then both $A' - u$ and $\phi(A) - u$ represent
$M \del u$ over $\GF(3)$ (by Proposition~\ref{prop:hom}).
Since representations are unique over $\GF(3)$ (\cite{BL76}),
and $\GF(3)$ has no non-trivial automorphisms,
by scaling we can assume that $A' - u = \phi(A) - u$.
Now there are nonsingular diagonal matrices $D_{1}$ and
$D_{2}$ such that $D_{1}(A' - v)D_{2} = \phi(A) - v$.
Just as in the proof of Lemma~\ref{GGK2.2}, we can prove that
$A' - v = \phi(A) - v$, up to scaling of the
column $u$.
The result follows.
\end{proof}

\begin{lem}\label{GGK2.3}
  Let $M$ be a matroid, and let $u$, $v$ be a deletion pair for $M$
such that $M\del u$ and $M\del v$ are near-regular.
Let $X$ be a basis of $M \del \{u,v\}$, and let $Y = E(M) \setminus X$.
Let $A$ be the $X\times Y$ \qalpha\dash matrix such that
$M[I|A-u] = M\del u$, $M[I|A-v]=M \del v$, and $A-u$ and $A-v$ are
near-unimodular.
Now assume that $X'\subseteq X$ and $Y'\subseteq Y - \{u,v\}$ are
such that
  \begin{enumerate}
    \item $u$, $v$ is a deletion pair for $M \con X' \del Y'$; and
    \item $M \con X' \del Y' \neq M[I|A - (X' \cup Y')]$.
  \end{enumerate}
Then $M\con X' \del Y'$ is not near-regular.
\end{lem}

\begin{proof}
Suppose that $M \con X' \del Y'$ is near-regular, and that $A'$ is
an $(X \setminus X') \times (Y \setminus Y')$ $\nreg$\dash matrix
that represents $M \con X' \del Y'$.
Deleting $u$ or $v$ from $A'$ produces a near-unimodular
matrix that represents $M \del u$ or $M \del v$ respectively.
But the same statements apply to $A - (X' \cup Y')$.
The uniqueness guaranteed by Lemma~\ref{GGK2.2} means that
$M[I|A'] = M[I|A - (X' \cup Y')]$, so we have a contradiction
to the hypotheses of the lemma.
\end{proof}

\begin{lem}
\label{lem:notQarep}
Let $M$ be an excluded minor for the class of near-regular
matroids such that $M$ is representable over
$\GF(3)$ and $\GF(4)$.
Then $M$ is not representable over $\qalpha$.
\end{lem}

\begin{proof}
Let \mcal{M} be the set of matroids representable over
$\GF(3)$, $\GF(4)$, and \qalpha.
We claim that this is precisely the class of near-regular
matroids.
Theorem~1.5 of~\cite{Whi97} shows that \mcal{M}
is exactly the set of matroids representable over both
$\GF(3)$ and $\GF(q)$, for some $q \in \{2, 3, 4, 5, 7, 8\}$.
It cannot be the case that $q=2$, for then \mcal{M} would
be the set of regular matroids.
Since \mcal{M} contains $U_{2,4}$ this is impossible.

Consider the matroid $\ag{2,3}$.
It is representable over the field $\field$
if and only if $\field$ contains a solution to
$x^{2} - x + 1 = 0$ (\cite[p.~515]{oxley}).
Since \qalpha\ contains no such solution, it
follows that \ag{2,3} is not $\qalpha$\dash representable,
and therefore does not belong to $\mcal{M}$.
However, \ag{2,3} is representable over
$\GF(3)$, $\GF(4)$, and $\GF(7)$ (since
$x = 3$ is a solution to $x^{2}-x+1=0$).
Thus $q$ cannot be equal to $3$, $4$, or $7$.
We conclude that $q$ is equal to either $5$ or $8$.
In either case Theorem~\ref{thm:nearregreps} implies that
\mcal{M} is the class of near-regular matroids,
as desired.
The result follows immediately.
\end{proof}

\section{Connectivity}\label{sec:machinery}

Much of this paper consists of recovering connectivity in situations where
it seems to have been lost. Our tool for this is the blocking sequence. 
Suppose that $M$ is a matroid on the ground set $E$.
We introduce a similar notation to that used for induced submatrices.
Suppose $E = B\cup Y$ where $B\cap Y = \emptyset$ and $B$ is a basis of $M$.
Let $Z$ and $Z'$ be subsets of $E$. Then $M_B[Z] := M \con (B\setminus Z) \del (Y\setminus Z)$, and $M_B-Z := M_{B}[E\setminus Z]$.
Moreover, $M_{B}[Z] - Z' = M_{B}[Z\setminus Z']$.

\begin{dfn}
Let $M$ be a matroid, $B$ a basis of $M$, and suppose that
$X$ and $Y$ are subsets of $E(M)$. Then
\begin{displaymath}
\lambda_B(X,Y) := \rank_{M\con (B\setminus Y)}(X\setminus B) + \rank_{M\con (B\setminus X)}(Y\setminus B).
\end{displaymath}
\end{dfn}
It is straightforward to verify that this is the same as the
function $\lambda_{B}(X,Y)$ employed in~\cite{GGK}.
Moreover, if $X$ and $Y$ are disjoint, then
\begin{displaymath}
\lambda_{B}(X,Y) = \rank_{M_{B}[X\cup Y]}(X) + \rank_{M_{B}[X\cup Y]}(Y)
- \rank(M_{B}[X\cup Y]),
\end{displaymath}
which is the usual connectivity function of $M_{B}[X \cup Y]$.
In particular, if $X$ and $Y$ partition $E(M)$, then
$\lambda_{B}(X,Y) = \rank_{M}(X) + \rank_{M}(Y) - \rank(M)$.
If $X$ and $Y$ are disjoint, then we say that
$(X,Y)$ is a \emph{$k$\dash separation} of $M_{B}[X\cup Y]$
if $|X|,|Y| \geq k$ and $\lambda_{B}(X,Y) < k$.

When $M$ is representable the following holds:
\begin{lem}\label{lem:matrixconn}
  Suppose $A$ is an $(X_1\cup X_2)\times(Y_1\cup Y_2)$ $\parf$\dash matrix
(where $X_1$, $X_2$, $Y_1$, and $Y_2$ are pairwise disjoint).
Let $M = M[I|A]$. Then
\begin{displaymath}
\lambda_{X_1\cup X_2}(X_1\cup Y_1, X_2\cup Y_2) =
\matrank(A[X_2,Y_1]) + \matrank(A[X_1,Y_2]). 
\end{displaymath}
\end{lem}

Let $M$ be a matroid on the ground set $E$, and let $B$ be a basis of $M$.
It is well-known that $G_{B}(M)$ is connected if and only if $M$
is connected.
A partition $(X,Y)$ of $E$ is a \emph{split} with respect to $B$ if $|X|, |Y| \geq 2$ and the edges in $G_B(M)$ that join vertices in $X$ to vertices in $Y$ induce a complete bipartite graph. Note that this bipartite graph need not span all vertices in either $X$ or $Y$.

\begin{prop}
\tup{\cite[{Proposition 4.11}]{GGK}}.
\label{GGK4.11}
  Let $M$ be a matroid, and suppose $B$ is a basis of $M$.
If $(X,Y)$ is a $2$\dash separation of $M$, then $(X,Y)$ is a split
with respect to $B$.
\end{prop}

Not every split corresponds to a $2$\dash separation:
\begin{prop}
\tup{\cite[{Proposition 4.12}]{GGK}}.
\label{GGK4.12}
  Let $B$ be a basis of the matroid $M$, and let $(X,Y)$ be a split with
respect to $B$.
Suppose $x_1y_1$ is an edge of $G_B(M)$ with $x_1 \in X$ and $y_1 \in Y$.
Then $(X,Y)$ is not a $2$\dash separation of $M$ if and only if there exist
$x_2 \in X$ and $y_2 \in Y$ such that $M_B[\{x_1,y_1,x_2,y_2\}] \cong U_{2,4}$.
\end{prop}

The following definitions and lemmas are directly from~\cite[Section 4]{GGK}, and will be presented here without proof.
There is some overlap with results due to Truemper~\cite{TruIII}, who also
gives a very detailed analysis of the structure of the resulting matrices when $M$ is representable.
\begin{dfn}
  Let $M$ be a matroid, and let $B$ be a basis of $M$.
Suppose that $(X,Y)$ is an exact $k$\dash separation of $M_B[X\cup Y]$. We say that
$(X,Y)$ is \emph{induced} if there exists a $k$\dash separation $(X',Y')$ of $M$ with $X \subseteq X'$ and $Y \subseteq Y'$.
\end{dfn}

\begin{dfn}\label{def:blockseq}
Suppose that $M$ is a matroid, and that $B$ is a basis of $M$.
Let $(X,Y)$ be a $k$\dash separation of $M_B[X\cup Y]$. A \emph{blocking sequence} for $(X,Y)$ is a sequence of elements $v_1, \ldots, v_p$ of $E(M)\setminus (X\cup Y)$ such that
\begin{enumerate}
  \item\label{bls:first} $\lambda_B(X, Y \cup v_1) = k$;
  \item\label{bls:middle} $\lambda_{B}(X\cup v_i, Y\cup v_{i+1}) = k$ for $i = 1, \ldots, p-1$;
  \item\label{bls:last} $\lambda_{B}(X\cup v_p,Y) = k$; and
  \item No proper subsequence of $v_1, \ldots, v_p$ satisfies the first three properties.
\end{enumerate}
\end{dfn}
The following proposition shows how useful blocking sequences are:
\begin{prop}
\tup{\cite[Theorem 4.14]{GGK}}.
\label{lem:blockseq}
  Let $M$ be a matroid, and suppose that $B$ is a basis of $M$.
Let $(X,Y)$ be an exact $k$\dash separation of $M_B[X\cup Y]$. Exactly one of the following holds:
  \begin{enumerate}
    \item There exists a blocking sequence for $(X,Y)$.
    \item $(X,Y)$ is induced.
  \end{enumerate}
\end{prop}
The first of the following propositions lists basic properties of blocking sequences; the next provides a means of shortening a given sequence.
\begin{prop}
\tup{\cite[Proposition~4.15~(\emph{i},~\emph{ii},~\emph{iv})]{GGK}}.
\label{GGK4.15}
Suppose that $M$ is a matroid on the ground set $E$, and that
$B$ is a basis of $M$.
Let $(X,Y)$ be an exact $k$\dash separation in $M_B[X\cup Y]$, and
suppose that $v_1, \ldots, v_p$ is a blocking sequence for $(X,Y)$.
Then the following hold:
\begin{enumerate}
  \item \label{it:4.15.i} For $1 \leq i \leq j \leq p$, $v_i, \ldots, v_j$ is a blocking sequence for the exact $k$\dash separation $(X \cup \{v_1, \ldots, v_{i-1}\}, Y \cup \{v_{j+1}, \ldots, v_p\})$ of $M_B[X \cup Y \cup \{v_1,\ldots, v_{i-1},v_{j+1}, \ldots, v_p\}]$.
  \item \label{it:4.15.ii}Let $x_1,x_2 \in X\cup Y$ be such that $x_1x_2$ is an edge of $G_B(M)$. Then $v_1, \ldots, v_p$ is a blocking sequence for the exact $k$\dash separation $(X,Y)$ of $M_{B\symdiff\{x_1,x_2\}}[X\cup Y]$.
  \item \label{it:4.15.iv}For $i = 1, \ldots, p-1$, $v_i\in B$ implies
$v_{i+1} \in E\setminus B$, and $v_i \in E \setminus B$ implies $v_{i+1} \in B$.
\end{enumerate}
\end{prop}

\begin{prop}
\tup{\cite[Proposition 4.16]{GGK}}.
\label{GGK4.16}
Let $M$ be a matroid. Suppose that $B$ is a basis of $M$, and that $(X,Y)$ is
an exact $k$\dash separation in $M_B[X\cup Y]$.
Let $v_1, \ldots, v_p$ be a blocking sequence for $(X,Y)$.
Then the following hold:
\begin{enumerate}
  \item \label{it:4.16.i}Suppose that $Y' \subseteq Y$ contains at least $k$ elements and that $\lambda_B(X,Y') = k-1$.
If $p>1$, then $v_1,\ldots,v_{p-1}$ is a blocking sequence for the exact $k$\dash separation $(X,Y'\cup v_p)$ of $M_B[X\cup Y' \cup v_p]$.
  \item \label{it:4.16.ii}Let $y \in Y$ be such that $v_py$ is an edge of $G_B(M)$, and $\lambda_B(X\cup y, Y) = k$. If $p > 1$, then $v_1, \ldots, v_{p-1}$ is a blocking sequence for the exact $k$\dash separation $(X,(Y \cup v_p) \setminus y\})$ of $M_{B\symdiff\{v_{p},y\}}[(X \cup Y \cup v_p) \setminus y]$.
  \item \label{it:4.16.iii}If $v_i$ has no neighbors in $X\cup Y$ in $G_B(M)$, then $1 < i < p$; $v_{i-1}v_{i}$ is an edge of $G_B(M)$; and $v_1,\ldots,v_{i-2}, v_{i+1},\ldots,v_p$ is a blocking sequence for the exact $k$\dash separation $(X,Y)$ of $M_{B\symdiff \{v_{i-1},v_{i}\}}[X\cup Y]$.
\end{enumerate}
\end{prop}

For $2$\dash separations more can be said. If $(X_1,Y_1)$ and $(X_2,Y_2)$ are both partitions of a set, then these partitions \emph{cross} if $X_i\cap Y_j \neq \emptyset$ whenever $i,j \in \{1,2\}$.
\begin{dfn}
  Let $M$ be a matroid, and suppose that $(X_1,Y_1)$ is a $2$\dash separation
of $M$. We say $(X_1,Y_1)$ is \emph{crossed} if there exists a
$2$\dash separation $(X_2,Y_2)$ of $M$ such that $(X_1,Y_1)$ and $(X_2,Y_2)$
cross. Otherwise $(X_1,Y_1)$ is \emph{uncrossed}.
\end{dfn}

\begin{prop}
\tup{\cite[Proposition 4.17]{GGK}}.
\label{GGK4.17}
  Let $B$ be a basis of the matroid $M$. Suppose that $(X_1,X_2)$ is
an uncrossed $2$\dash separation of $M_B[X_1\cup X_2]$, and let
$v_1, \ldots, v_p$ be a blocking sequence for $(X_1,X_2)$.
Let $(Y_1,Y_2)$ be a $2$\dash separation of
$M_B[X_1\cup X_2 \cup \{v_1,\ldots,v_p\}]$. Then, for some
$i,j \in \{1,2\}$, $X_i \cup \{v_1, \ldots, v_p\} \subseteq Y_j$.
\end{prop}

\begin{prop}
\tup{\cite[Proposition 4.18]{GGK}}.
\label{GGK4.18}
    Let $M$ be a matroid, and let $B$ be a basis of $M$.
Suppose that $(X_1,X_2)$ is an uncrossed $2$\dash separation in
$M_B[X_1\cup X_2]$, and let $v \in E(M)\setminus(X_1\cup X_2)$ be such
that $\lambda_B(X_1\cup v, X_2) = 2$.
If $(Y_1,Y_2)$ is a $2$\dash separation of $M_B[X_1\cup X_2 \cup v]$
such that $X_2\subseteq Y_2$, then $v \in Y_2$.
\end{prop}

\begin{prop}
\tup{\cite[Corollary 4.19]{GGK}}.
\label{GGK4.19}
Suppose $B$ is a basis of the matroid $M$.
If $(X_1,X_2)$ is the unique $2$\dash separation in $M_B[X_1\cup X_2]$, and $v_1,\ldots, v_p$ is a blocking sequence for $(X_1,X_2)$, then $M_B[X_1\cup X_2 \cup \{v_1,\ldots,v_p\}]$ is $3$\dash connected.
\end{prop}

\section{The reduction}\label{sec:reduction}

This section contains the core of the proof of Theorem~\ref{thm1}.
We reduce the proof to a finite case-analysis by showing that any
excluded minor for the class of near-regular matroids
has at most eight elements. This part of the proof follows the arguments
in~\cite{GGK} very closely. Deviations necessarily occur when the nature of $\GF(4)$ comes into play. This happens in the case $k = 0$ of Claim~\ref{GGK15} (which is~($15$) in~\cite{GGK}) and from Claim~\ref{GGK20} (which is~($20$) in~\cite{GGK}) to the end. All other differences are largely cosmetic: for example, rather than work with the bipartite graphs associated with matrices, we choose to work with the matrices themselves.

We denote the simplification or cosimplification of a
matroid $M$ by $\si(M)$ or $\co(M)$.
Suppose that the matroid $M$ has $E$ as its ground set
and \mcal{B} as its set of bases.
Let $B$ be a basis of $M$, and suppose that $x \in E$.
Then $\nigh_{B}(x)$ denotes the set of vertices of $G_{B}(M)$
that are adjacent to $x$.
Thus
\begin{displaymath}
\nigh_{B}(x) = \big\{y \in E\,\,\big|\,\,B \symdiff \{x,y\} \in \mcal{B}\big\}.
\end{displaymath}

\begin{thm}\label{thm:finitesize}Let $M$ be an excluded minor for the class of near-regular matroids other than $\agde$ or  $(\agde)^*$.
Then $\rank(M) \leq 4$ and $\corank(M) \leq 4$.
\end{thm}

\begin{proof}
  Suppose the theorem is false. Let $M$ be an excluded minor for the class of near-regular matroids on the ground set $E$, such that $\rank(M) > 4$ or $\corank(M)> 4$, and suppose that
$M$ is isomorphic to neither $\agde$ nor $(\agde)^*$.
Lemmas~\ref{lem:ternary} and~\ref{lem:3c} imply that $M$ is ternary and
$3$\dash connected.
If $M$ is not $\GF(4)$\dash representable, then it is an excluded minor
for $\GF(4)$\dash representability.
But none of the matroids in Theorem~\ref{thm4} is a counterexample
to Theorem~\ref{thm:finitesize}, so this is a contradiction.
Thus $M$ is also $\GF(4)$\dash representable.

Lemma~\ref{lem:delpair} says that for some $M' \in \{M,M^*\}$,
there is a deletion pair $u$, $v$ of $M'$, and that
$M'\del \{u,v\}$ contains a $3$\dash connected nonbinary minor
of size at least $|E|-4$.
  \begin{assumption}\label{GGK1}
    $M'$, $u$, and $v$ have been chosen so that
$|\co(M'\del \{u,v\})|$ is as large as possible.
  \end{assumption}
Lemma~\ref{lem:duality} implies that $M$ is a counterexample
to the theorem if and only if $M^{*}$ is, so henceforth
we relabel $M'$ with $M$.

Since $\{u,v\}$ is coindependent, there is a basis $B$ of
$M$ that contains neither $u$ nor $v$.
Define $Y := E\setminus B$. Lemma~\ref{GGK2.2} supplies
a $B\times Y$ \qalpha\dash matrix $A$ with entries in $\nreg$
such that
$M[I\mid A-u] = M \del u$, $M[I \mid A-v] = M \del v$, and
both $A-u$ and $A-v$ are near-unimodular.
Let $N$ be the matroid represented over \qalpha\ by $A$.
Thus $N \del u = M \del u$ and $N \del v = M \del v$.
Since $M$ is representable over both $\GF(3)$ and $\GF(4)$,
Lemma~\ref{lem:notQarep} implies that $M$ is not
\qalpha\dash representable.
Hence $M \neq N$.
There is a set $B'$ that is a basis in exactly one of $M$ and $N$,
and such a basis must contain $\{u,v\}$.
By extending $B' \setminus \{u,v\}$ to a basis of $M\del \{u,v\}$
we see that the following claim holds.
\begin{claim}
\label{GGK2}
    Let $B'$ be a set that is a basis in exactly one of $M$ and $N$.
There is a basis $B''$ of $M\del \{u,v\} = N\del \{u,v\}$ such that
$B'\setminus B'' = \{u,v\}$.
\end{claim}
Let $B'$ and $B''$ be as in Claim~\ref{GGK2}.
By Lemma~\ref{lem:pivot}, we can pivot, and assume that
$A$ is a $B'' \times (E \setminus B'')$ matrix.
Henceforth we relabel $B''$ with $B$ and $E \setminus B'$
with $Y$.
Note that, although $M \neq M[I|A]$, the fact that
$M\del u$ and $M \del v$ are represented by
$A - u$ and $A - v$ respectively means that $G_{B}(M) = \bip(A)$.

If $B_{1}$ is a basis of $M\del \{u,v\} = N\del\{u,v\}$,
and $B_{2}$ is a basis of exactly one of $M$ and $N$, then
we say that $B_{1}\symdiff B_{2}$ is a \emph{distinguishing set}
with respect to $B_{1}$.
Define $\{a,b\} := B\setminus B'$.
Then $\{a,b,u,v\}$ is a distinguishing set with respect to $B$.
  \begin{claim}\label{cl:distsetindependent}
    $B'$ is a basis of $N$.
  \end{claim}
  \begin{subproof}
Suppose that the claim is false.
Then the determinant of $A[\{a,b,u,v\}]$, evaluated over
$\qalpha$, is equal to zero.
Let $\phi$ be the unique homomorphism from $\nreg$ to $\GF(3)$.
Proposition~\ref{prop:hom} implies that the determinant of
$\phi(A)[\{a,b,u,v\}]$, evaluated over $\GF(3)$, is also zero.
Thus $B'$ is not a basis of $M[I|\phi(A)]$.
But Lemma~\ref{lem:ternaryhom} says that $M[I|\phi(A)]$ is
$M$, so we have a contradiction.
  \end{subproof}

  \begin{claim}\label{cl:distsetcycle}
    $\bip(A[\{a,b,u,v\}])$ is a cycle.
  \end{claim}
  \begin{subproof}
Suppose that the claim fails.
Then there is some zero entry in $A[\{a,b,u,v\}]$, and
hence in $\phi(A)[\{a,b,u,v\}]$.
Since $B'$ is not a basis of $M$, the determinant of
$\phi(A)[\{a,b,u,v\}]$ evaluated over $\GF(3)$ must be zero.
This implies that $\phi(A)[\{a,b,u,v\}]$ must contain a zero
row or column.
However, as $\phi$ takes no nonzero element to zero, this implies that
$A[\{a,b,u,v\}]$ has a zero row or column, which is a contradiction
as $B'$ is a basis of $N$.
  \end{subproof}

The remainder of the proof consists of refining the choices of $u$, $v$, $B$, $a$, and $b$, always relabeling as necessary so that $\{a,b,u,v\}$ remains a distinguishing set. For that, we need to restrict our pivots. A pivot over $xy$, where $x \in B$ and $y \in Y\setminus \{u,v\}$, is \emph{allowable} if
  \begin{enumerate}
    \item $x \in \{a,b\}$;
    \item $A_{ay} = A_{by} = 0$; or
    \item $A_{xu} = A_{xv} = 0$.
  \end{enumerate}
  In the first case, $\{a,b,u,v\}\symdiff\{x,y\}$ is a distinguishing set with respect to $B\symdiff \{x,y\}$. This is obvious, since $(B\symdiff \{x,y\}) \symdiff (\{a,b,u,v\}\symdiff\{x,y\}) = B \symdiff \{a,b,u,v\} = B'$.
Suppose that $A_{ay} = A_{by}=0$.
Then the determinant of $A[\{a,b,u,v,x,y,\}]$ evaluated over \qalpha,
is equal to $A_{xy}$ times the determinant of $A[\{a,b,u,v\}]$, which is
nonzero as $B'$ is a basis of $N$.
It follows that $B' \symdiff \{x,y\}$ is a basis of $N$.
On the other hand, the determinant of
$\phi(A)[\{a,b,u,v,x,y\}]$ evaluated over $\GF(3)$, is
equal to $\phi(A)_{xy}$ times the determinant of
$\phi(A)[\{a,b,u,v\}]$, which is zero.
Thus $B' \symdiff \{x,y\}$ is not a basis of
$M[I|\phi(A)] = M$.
Therefore
\begin{displaymath}
(B \symdiff \{x,y\}) \symdiff (B' \symdiff \{x,y\}) = \{a,b,u,v\}
\end{displaymath}
is a distinguishing set with respect to the basis
$B \symdiff \{x,y\}$.
A similar argument shows that if $A_{xu} = A_{xv}=0$, then
$\{a,b,u,v\}$ is a distinguishing set with respect to
$B \symdiff \{x,y\}$.

Since $M\del\{u,v\} = N\del\{u,v\}$ is nonbinary, there is some
$C \subseteq B \cup (Y\setminus\{u,v\})$ such that
$A[C]$ is a twirl, by Lemma~\ref{GGK4.3}.
If $x$ is a vertex of $\bip(A-\{u,v\})$, then
$d(x,C)$ denotes the length of a shortest (possibly empty)
path in $\bip(A-\{u,v\})$ that joins $x$ to a vertex in $C$.
  \begin{assumption}\label{GGK3}
    Subject to~{\rm\ref{GGK1}}, we choose $u$, $v$, $B$, $a$, $b$, and
$C$ so that $(|C|, d(a,C), d(b,C))$ is lexicographically minimal.
  \end{assumption}

  \begin{claim}\label{GGK4}
  If $x \in E \setminus C$, then
$|\nigh_{B}(x) \cap C| \leq 2$.
If $a \notin C$, then $|\nigh_{B}(a) \cap C| \leq 1$.
If $b \notin C$ and $|\nigh_{B}(b) \cap C| = 2$, then
$a \in C$.
  \end{claim}
  \begin{subproof}
  Suppose that $x \in E \setminus C$ and that
$|\nigh_{B}(x) \cap C| \geq 2$.
Lemma~\ref{GGK4.4} implies that we can find a
twirl $C'$ in $C \cup x$.
If $|\nigh_{B}(x) \cap C| \geq 3$, then
$|C'| < |C|$, and this contradicts~\ref{GGK3}, so
$|\nigh_{B}(x) \cap C| =2$ and $|C'| = |C|$.
Now we suppose that $x = a$.
Then $0= d(a,C') < d(a,C)$, and we have
a contradiction to~\ref{GGK3} that proves the second statement.
Finally, if $x = b$, then
$d(a,C') \leq d(b,C') = 0$, and this proves the third statement.
  \end{subproof}

  \begin{claim}\label{GGK5}
    $|C| = 4$.
  \end{claim}
  \begin{subproof}
    Suppose $|C| \geq 6$, and let $x,y \in C$ be such that $A_{xy} \neq 0$. A pivot over $xy$ is not allowable, because otherwise, by Proposition~\ref{GGK4.5}, a shorter twirl can be found, contradicting~\ref{GGK3}. It follows that $\{a,b\}\cap C = \emptyset$. Therefore Claim~\ref{GGK4} implies that
    \begin{displaymath}
|\nigh(a) \cap C|,\, |\nigh(b)\cap C| \leq 1.
    \end{displaymath}
Hence there is an edge $xy$ in $A[C]$ such that neither
$x$ nor $y$ is adjacent to either $a$ or $b$.
Thus the pivot on $xy$ is allowable, and we have a contradiction
that proves the claim.
  \end{subproof}

Now we split the proof into three different cases:
  \begin{enumerate}
\item $a,b \in C$;
\item $a \in C$ and $b \notin C$; and
\item $a,b \notin C$.
  \end{enumerate}
By using Claim~\ref{GGK4}, and by scaling $A$, we can assume that
in cases~(\emph{i}),~(\emph{ii}), and~(\emph{iii}) (respectively),
$A[C\cup \{a,b,u,v\}]$ is equal to $A_{1}$, $A_{2}$, or
$A_{3}$ (respectively), where these matrices are shown in Table~\ref{tab1}.
Here elements in $C\setminus\{a,b,u,v\}$ are labeled
with elements from $\{1,2,3,4\}$.
A star marks an unknown entry (possibly equal to zero);
entries labelled by $g$, $q$, and $r$ are not equal to $0$ or $1$.
\begin{table}[htb]
  \begin{align*}
    A_1 = \kbordermatrix{ & 1 & 2 & u & v\\
                          a & 1 & 1 & 1 & 1\\
                          b & q & r & 1 & g
                          }\\
    A_2 = \kbordermatrix{ & 1 & 2 & u & v\\
                          3 & q & 1 & * & *\\
                          a & 1 & 1 & 1 & 1\\
                          b & * & * & 1 & g
                          }\\
    A_3 = \kbordermatrix{ & 3 & 1 & u & v\\
                          4 & q & 1 & * & *\\
                          2 & 1 & 1 & * & *\\
                          a & 0 & * & 1 & 1\\
                          b & * & * & 1 & g
                          }
  \end{align*}
  \caption{$A[C \cup \{a,b,u,v\}]$ is one of these matrices.}
  \label{tab1}
\end{table}
  In the remainder of the proof we deal with these cases one by one. Most of the work will be in the second case, which we will save for last.
  \begin{claim}\label{GGK9}
    If $A_{ay} \neq 0$ and $A_{by} \neq 0$ for some $y \in Y\setminus\{u,v\}$ then $A_{by}/A_{ay} \not\in\{1,g\}$.
  \end{claim}
  \begin{subproof}
Suppose that the claim fails.
Then, after pivoting on $ay$, and relabeling $y$ with $a$,
we see that $A[\{a,b,u,v\}]$ contains a zero entry.
But pivoting on $ay$ is allowable, so $\{a,b,u,v\}$ remains a
distinguishing set.
Now we can deduce a contradiction to
Claim~\ref{cl:distsetcycle}.
  \end{subproof}

We dispose of the first case very easily.

  \begin{claim}
    $b \not \in C$.
  \end{claim}
  \begin{subproof}
Suppose otherwise, so that $a,b \in C$, and $A[\{a,b,1,2,u,v\}] =A_{1}$.
Claim~\ref{GGK9} implies that $r \notin \{1,g\}$, and
$r\notin \{0,q\}$ as $A[\{a,b,1,2\}]$ is a twirl.
It follows that
$M[I| A[\{a,b,1,2,u\}]] \cong U_{2,5}$, which contradicts
the fact that $M\del v$ is ternary.
  \end{subproof}

Note that if $\{u,v\}\subseteq Z \subseteq E$, then $\{u,v\}$ is necessarily
coindependent in $M_{B}[Z]$, since neither $u$ nor $v$ is in
$B$.
Now the following result is an easy consequence of Lemma~\ref{GGK2.3}:
  \begin{claim}\label{GGK7}
    Let $Z \subseteq E$ be such that $\{u,v\} \subseteq Z$,
$M_B[Z]-u$, $M_B[Z]-v$, $M_B[Z]-\{u,v\}$ are stable, $M_B[Z]-\{u,v\}$ is connected and nonbinary, and $M_B[Z] \neq N_{B}[Z]$. Then $Z = E$.
  \end{claim}
  Now we dispense with the third case:
  \begin{claim}\label{GGK10}
    $a \in C$.
  \end{claim}
  \begin{subproof}
    Suppose this is false.  Let $Z := \{a,b,u,v,1,2,3,4\}$,
so $A[Z] = A_{3}$.
Our first step is to recover some connectivity.
    \begin{sclaim}\label{GGK6}
      $A_{a1} \neq 0$.
    \end{sclaim}
    \begin{subproof}
      Suppose otherwise. Then $d(a,C) > 1$. Since $M \del \{u,v\}$,
and hence $\bip(A-\{u,v\})$, is connected, there is a path from $a$ to
$C$ in $\bip(A-\{u,v\})$.
Let $x_1, \ldots, x_k$ be the internal vertices of a shortest path
from $a$ to $C$. Then $x_k$ has exactly one neighbor in $C$, because otherwise Lemma~\ref{GGK4.4} implies the existence of a twirl $A[C]'$, where $x_{k} \in C'$, and $C' \subseteq C \cup \{x_{k}\}$.
Then $|C'| = 4$, and $d(a,C')<d(a,C)$, contradicting~\ref{GGK3}. Let $x$ be the unique neighbor of $x_k$ in $C$. Let $y \in C$ be a neighbor of $x$ and let $z \in C$ be the other neighbor of $y$. Since $d(b,C) \geq d(a,C) > 1$, pivoting on $xy$ is allowable. But after this pivot, $x_k$ is adjacent to both $y$ and $z$, so we have reduced to a previous case and we can again derive a contradiction.
    \end{subproof}

    \begin{sclaim}\label{GGK10.1}
      $A_{a3} = A_{b3} = 0$.
    \end{sclaim}
    \begin{subproof}
We have already assumed that $A_{a3} = 0$, by virtue of Claim~\ref{GGK4}.
The same claim implies that $|\nigh(b)\cap C| \leq 1$.
Suppose $A_{b3} \neq 0$. Then $A_{b1} = 0$.
In this case $M_B[Z]-\{u,v\}$ is connected (since
$\bip(A[Z]-\{u,v\})$ is connected), nonbinary (because
it has a whirl-minor), and stable (since it is a $2$\dash element
coextension of a whirl).
By examining $\bip(A[Z]-\{u,2\})$ and using
Proposition~\ref{GGK4.11}, it is easy to see that $M_B[Z]- \{u,2\}$ is
$3$\dash connected. Thus $M_B[Z]-u$ is stable.
A similar argument shows that $M_B[Z]-v$ is stable.
Since $\{a,b,u,v\}$ is a distinguishing set, Claim~\ref{GGK7}
implies that $E = Z$, contradicting the assumption that
$\rank(M) \geq 5$ or $\corank(M) \geq 5$.
    \end{subproof}

    \begin{sclaim}\label{GGK10.2}
      $(\{a,b,1\}, \{2,3,4\})$ is an induced $2$\dash separation of $M_{B}-\{u,v\}$.
    \end{sclaim}
    \begin{subproof}
Proposition~\ref{GGK4.11} and Claim~\ref{GGK10.1} imply that $(\{a,b,1\},\{2,3,4\})$ is a $2$\dash separation of $M_{B}[Z]-\{u,v\}$. Suppose that it is not induced. Then there is a blocking sequence $v_1, \ldots, v_p$.
We will assume that, subject to~\ref{GGK1} and~\ref{GGK3},
$u$, $v$, $B$, $a$, $b$, and $C$ have been chosen so that $p$ is
as small as possible.

First suppose that $v_p$ labels a column of $A$.
By Definition~\ref{def:blockseq}, $(\{a,b,1,v_p\},\{2,3,4\})$ is not a $2$\dash separation in $M_{B}[(Z \setminus \{u,v\}) \cup v_p]$.
In the graph $\bip(A[Z]-\{u,v\})$, $1$ is the
only vertex in $\{a,b,1\}$ that is adjacent to a vertex in
$\{2,3,4\}$.
Thus it follows without
difficulty from Proposition~\ref{GGK4.12} that $v_p$ is adjacent to
either $2$ or $4$ in $\bip(A[(Z \setminus \{u,v\}) \cup v_p])$.
Now, by pivoting on either $A_{32}$ or $A_{34}$ (and relabeling),
we can assume $v_{p}$ is adjacent to both $2$ and $4$.
(Note that this pivot is allowable.)
Thus $(\{a,b,1,v_p\},\{2,3,4\})$ is a split, but not a
$2$\dash separation.
It follows from Proposition~\ref{GGK4.12} that $A[\{1,2,v_p,4\}]$ is a twirl. We can now replace $3$ with $v_p$. If $p = 1$ then $v_p$ is adjacent to $a$ or $b$, contradicting Claim~\ref{GGK10.1}. If $p > 1$, then by taking $Y'=\{2,4\}$, we see that Proposition~\ref{GGK4.16}~\eqref{it:4.16.i} implies $v_1, \ldots, v_{p-1}$ is a blocking sequence for $(\{a,b,1\},\{2,v_p,4\})$.
This contradicts our assumption of minimality, so we are done.

Now suppose $v_p$ labels a row. Again, $(\{a,b,1,v_p\},\{2,3,4\})$ is not a $2$\dash separation in $M_{B}[(Z \setminus \{u,v\})\cup v_p]$. Hence $A_{v_p3} \neq 0$. Using an allowable pivot if necessary we also have $A_{v_p1} \neq 0$. By Lemma~\ref{GGK4.4} either $A[\{v_p,1,2,3\}]$ or $A[\{v_p,1,3,4\}]$ is a twirl. By relabeling we may assume the latter holds. We now replace $2$ by $v_p$. Since $(\{a,b,1\},\{v_p,2,3,4\})$ is a $2$\dash separation, $p>1$. But Proposition~\ref{GGK4.16}~\eqref{it:4.16.i} implies $v_1, \ldots, v_{p-1}$ is a blocking sequence for $(\{a,b,1\},\{v_p,3,4\})$, and we again have a contradiction to minimality.
    \end{subproof}

    \begin{sclaim}\label{GGK10.3}
      $M_{B}-\{a,b,u,v\}$ is $3$\dash connected, and $1$ is the only neighbor of
$a$ and $b$ in $\bip(A-\{u,v\})$.
    \end{sclaim}
    \begin{subproof}
      By the previous claim, $M_{B}-\{u,v\}$ has a $2$\dash separation $(Z_1, Z_2)$ with $\{a,b,1\}\subseteq Z_1$. By our choice of $u$ and $v$,
$M\del\{u,v\}$ contains a $3$\dash connected minor on at least
$|E|-4$ elements.
This means that $Z_1$ is equal to $\{a,b,1\}$. Since $A[\{1\},\{2,4\}]$ is nonzero, it follows from Lemma~\ref{lem:matrixconn} that $A[\{a,b\},Y-\{1,u,v\}]$ must be the zero matrix.
Thus $a$ and $b$ can have no neighbor in $\bip(A-\{u,v\})$ other
than $1$.
However, $M_{B}-\{u,v\}$ is connected, so both $a$ and $b$ are adjacent
to $1$.
Now we see that $\co(M\del\{u,v\}) = M_{B}-\{u,v,a,b\}$, so we are done.
    \end{subproof}

    \begin{sclaim}\label{cl:inbetween103and104}
      $A_{2u}$ and $A_{2v}$ are not both equal to zero. Likewise, $A_{4u}$ and $A_{4v}$ are not both equal to zero.
    \end{sclaim}
    \begin{subproof}
      If $A_{2u} = A_{2v} = 0$, then a pivot over $12$ is allowable. But after performing this pivot, we see that $|\nigh(a)\cap C| = 2$, and this contradicts
Claim~\ref{GGK4}.
The same argument shows that either $A_{4u}$ or $A_{4v}$ is nonzero.
    \end{subproof}

    \begin{sclaim}\label{GGK10.4}
      If $b' \in \{a,b\}$ and $v' \in \{u,v\}$ then $M_{B}-\{b',v'\}$ is $3$\dash connected.
    \end{sclaim}
    \begin{subproof}
 Without loss of generality, we can assume that $b' = b$ and $v' = v$.
It follows from Claim~\ref{GGK10.3} that $(\{a,1\},E\setminus\{a,b,u,v,1\})$ is the unique $2$\dash separation of $M_{B}-\{b,u,v\}$. Since $A_{au} \neq 0$, it follows that $(\{a,1\},E\setminus\{a,b,v,1\})$ is not a $2$\dash separation in $M_{B}-\{b,v\}$. Suppose now, that $(\{a,u,1\},E\setminus \{a,b,u,v,1\})$ is a $2$\dash separation in $M_{B}-\{b,v\}$. Since the only neighbors of $b$ in $\bip(A)$ are
$u$, $v$, and $1$, we deduce that $(\{a,b,u,1\},E\setminus\{a,b,u,v,1\})$ is a $2$\dash separation in $M_{B}-v$. But Claim~\ref{GGK9} implies that $A[\{a,b,u,1\}]$ is a twirl. Since $A[\{1,2,3,4\}]$ is a twirl, this contradicts the fact that $M \del v$ is stable. Thus $(\{a,u,1\},E\setminus \{a,b,u,v,1\})$ is not
a $2$\dash separation of $M_{B}-\{b,v\}$, and it follows that $u$ is a blocking sequence for the $2$\dash separation $(\{a,1\},E\setminus\{a,b,u,v,1\})$.
Proposition~\ref{GGK4.19} implies that $M_{B}-\{b,v\}$ is $3$\dash connected,
as desired.
    \end{subproof}

    \begin{sclaim}\label{newclaim1}
    $M\con a = N \con a$, and $M \con b = N \con b$.
    \end{sclaim}
    \begin{subproof}
Claim~\ref{GGK10.4} says that $M_{B}-\{a,u\}$, and $M_{B}-\{a,v\}$
are $3$\dash connected, and therefore stable.
Since $M_{B} - \{a,b,u,v\}$ is $3$\dash connected by
Claim~\ref{GGK10.3}, and $b$ is adjacent to $1$ in $G(A-\{a,u,v\})$,
it follows that $M_{B} - \{a,u,v\}$ is connected and stable.
It is nonbinary since it contains a whirl-minor.
Now, if $M_{B}[E\setminus a] \neq N_{B}[E \setminus a]$,
then Claim~\ref{GGK7} implies that $E \setminus a = E$.
This contradiction shows that $M / a = N / a$.
The same argument shows that $M / b = N / b$.
    \end{subproof}

    \begin{sclaim}\label{newclaim2}
    $a$, $b$ is a deletion pair of $M^{*}$, and
$M^{*}\del \{a,b\}$ contains a $3$\dash connected
nonbinary minor on at least $|E|-4$ elements.
    \end{sclaim}
    \begin{subproof}
Certainly $\{a,b\}$ is independent in $M$.
Claim~\ref{GGK10.3} implies that $M \con \{a,b\} \del \{u,v\}$
is $3$\dash connected.
It follows that $M \con \{a,b\}$ is stable.
Similarly, Claim~\ref{GGK10.4} shows that $M \con a$ and $M \con b$
are stable.
Moreover, Claim~\ref{GGK10.3} asserts that both
$M\con a \del u$ and $M \con a \del v$ are $3$\dash connected.
Thus both $M\con \{a,b\}\del u$ and $M \con \{a,b\} \del v$,
and hence both $\bip(A-\{a,b,u\})$ and $\bip(A-\{a,b,v\})$,
are connected.
This means that $\bip(A-\{a,b\})$ is connected, and therefore
so is $M \con \{a,b\}$.
Clearly $M/\{a,b\}$ is nonbinary, for $A[\{1,2,3,4\}]$
is a twirl.
The second part of the claim follows because $M \con \{a,b\} \del \{u,v\}$
is $3$\dash connected.
This completes the proof.    
    \end{subproof}

Claim~\ref{GGK10.3} implies that $\{a,b,1\}$ is a series class
in $M \del \{u,v\}$, and that $M \con \{a,b\} \del \{u,v\}$
is $3$\dash connected.
Therefore $\co(M\del \{u,v\}) \cong M \con \{a,b\} \del\{u,v\}$,
so $|E(\co(M \del \{u,v\}))| = |E|-4$.
Now~\ref{GGK1} implies that
$|E(\si(M\con \{a,b\}))| \leq |E|-4$.
The fact that $M \con \{a,b\} \del\{u,v\}$ is
$3$\dash connected implies that $u$ and $v$ are
either loops or in parallel pairs in $M \con \{a,b\}$,
and that $|E(\si(M\con \{a,b\}))| = |E|-4$.
Now we choose $M^{*}$ instead of $M$, and
$a$ and $b$ instead of $u$ and $v$.
The arguments of this paragraph show that~\ref{GGK1} is
still satisfied.

If we let $B_{0} = E \setminus B$, then $B_{0}$ is a basis of
$M^{*}$ that avoids $\{a,b\}$, and
$(B_{0}\setminus \{u,v\}) \cup \{a,b\} = E \setminus B'$ is a basis
of $N^{*}$, but not of $M^{*}$.
Now $A^{T}$, the transpose of $A$, is a $B_{0} \times (E \setminus B_{0})$
\qalpha\dash matrix that represents $N^{*}$.
Claim~\ref{newclaim1} shows that $A^{T} - a$ and
$A^{T} - b$ represent $M^{*} \del a = N^{*} \del a$
and $M^{*} \del b = N^{*} \del b$ respectively.
Moreover, $\bip(A^{T})$ is equal to $\bip(A)$, so
$A^{T}[C]$ is a twirl.
We substitute $u$ and $v$ for $a$ and
$b$, and $B_{0}$ for $B$.
The arguments above show that~\ref{GGK3} is still satisfied.
Thus we can repeat the arguments of Claim~\ref{GGK10.1}
and show that either $A^{T}_{u2} = A^{T}_{v2} = 0$,
or $A^{T}_{u4} = A^{T}_{v4} = 0$.
But this contradicts Claim~\ref{cl:inbetween103and104},
and completes the proof of Claim~\ref{GGK10}.
  \end{subproof}

  The remainder of the proof deals with the second case, in which
$A[C\cup \{a,b,u,v\}]= A_2$. Let $x_0, \ldots, x_{k+1}$ be the vertices of a shortest path from $b$ to $C$ in $\bip(A-\{u,v\})$, with $x_0 = b$ and $x_{k+1} \in C$.
  \begin{claim}\label{GGK11}
    $d(b,C) = k+1$ is odd.
  \end{claim}
  \begin{subproof}
    Suppose not. Then $x_k$ labels a column of $A$. Assume first that $A_{ax_k} \neq 0$. By pivoting over $a1$, if necessary, we may assume that $A_{3x_k} \neq 0$ as well. But then Lemma~\ref{GGK4.4} implies that $A[\{1,2,3,a,x_k\}]$ contains a twirl using $x_k$. This contradicts the minimality of $d(b,C)$ in~\ref{GGK3}. Therefore $A_{ax_k} = 0$, and hence $A_{3x_k} \neq 0$. Let $Z := \{u,v,a,1,2,3,x_0,\ldots,x_k\}$.

Note that $\bip(A[Z]-\{u,v,1\})$ is a path with $x_{0} = b$
and $a$ as its end vertices.
Since $A[\{a,b\},\{u,v\}]$ contains no zero entries, it follows that
$\bip(A[Z]- \{u,1\})$ and $\bip(A[Z]- \{v,1\})$
both contain a spanning cycle.
Moreover, it is not difficult to see that
neither of these graphs contains a split.
Proposition~\ref{GGK4.11} implies that
$M_{B}[Z]-\{u,1\}$ and $M_{B}[Z]-\{v,1\}$ are $3$\dash connected.
Hence $M_{B}[Z]-u$ and $M_{B}[Z]-v$ are both stable.
Furthermore, $M_{B}[Z]-\{u,v\}$ is clearly connected,
and nonbinary, as it contains $M_{B}[C]$ as a minor.
Since $\bip(A[Z]- \{u,v\})$ contains
a single cycle, namely $\bip(A[C])$, and $\bip(A[Z]- \{u,v,1\})$
is a path, it follows that by repeatedly
simplifying and cosimplifying $M_{B}[Z]-\{u,v\}$, we
eventually reduce to a whirl.
This implies that $M_{B}[Z]-\{u,v\}$ is stable.
As $Z$ contains a distinguishing set, Claim~\ref{GGK7} now
implies that $Z = E$.

We wish to prove that $u$, $1$ is a deletion pair of
$M$.
Certainly $\{u,1\}$ is coindependent in $M$.
We have already proved that $M\del \{u,1\}$ is $3$\dash connected.
Therefore $M \del \{u,1\}$, $M\del u$, and $M \del 1$ are all
stable.
It remains to show that $M\del \{u,1\}$ is nonbinary.
We noted that $\bip(A[Z]-\{u,1\}) = \bip(A-\{u,1\})$ contains a
spanning cycle.
Thus there is an induced cycle $C'$ in
$\bip(A-\{u,1\})$ that contains the edge $bv$.
We can assume that $A$ has been scaled in such a way that
$A_{e} = 1$ for every edge $e \in C'$ other than $bv$.
(Note that this is compatible with our assumption that
$A[C\cup \{a,b,u,v\}]=A_{2}$.)
Now $A_{bv} = g$ is not equal to one, for
$A[\{a,b,u,v\}]$ has nonzero determinant over \qalpha.
Suppose that $g = -1$.
Then $\phi(A)[\{a,b,u,v\}]$ has nonzero determinant
over $\GF(3)$.
But this implies that $B'$
is a basis of $M = M[I|\phi(A)]$, and this is a contradiction.
Therefore $g$ is neither $1$ nor $-1$, so
$A[C']$ has nonzero determinant.
It follows that $M_{B}[C'] = N_{B}[C']$ is a whirl.
Thus $M \del \{u,1\}$ is nonbinary, as desired.

We have shown that $u$, $1$ is a deletion pair.
Moreover, $M \del \{u,1\}$ is $3$\dash connected, so
$M \del \{u,1\}$ certainly contains a $3$\dash connected
nonbinary minor on at least $|E|-4$ elements.
But $d(b,C) > 1$, so $b$ is a degree-one vertex of $\bip(A-\{u,v\})$, and
hence $M \del \{u,v\}$ is not $3$\dash connected.
Thus
\begin{displaymath}
|E(\co(M\del \{u,1\}))| > |E(\co(M\del \{u,v\}))|,
\end{displaymath}
and we have a contradiction to~\ref{GGK1}.
This completes the proof of Claim~\ref{GGK11}.
  \end{subproof}

  It follows from Claim~\ref{GGK11} that $x_k$ labels a row, and hence either $A_{x_k1}\neq 0$ or $A_{x_k2} \neq 0$. By pivoting over $a1$ or $a2$ as needed, we assume that both are nonzero. If $k > 2$, then the pivot over $x_2x_3$ is
allowable, and such a pivot reduces $d(b,C)$, contradicting~\ref{GGK3}. Thus $k\in \{0,2\}$. Likewise, $A[\{a,1,2,x_k\}]$ is not a twirl, because otherwise replacing $3$ by $x_k$ would reduce $d(b,C)$. It follows that, by scaling, we can assume that
$A[\{a,1,2,3,x_0,\ldots,x_k,u,v\}]$ is one of the following matrices:
\begin{displaymath}
    \kbordermatrix{ & 1 & 2 & u & v\\
                  3 & q & 1 & * & *\\
                  a & 1 & 1 & 1 & 1\\
            x_0 = b & r & r & 1 & g}, \qquad
    \kbordermatrix{ & 1 & 2 & x_1 & u & v\\
                  3 & q & 1 & 0   & * & *\\
                  a & 1 & 1 & 0   & 1 & 1\\
                x_2 & 1 & 1 & 1   & * & *\\
            x_0 = b & 0 & 0 & r   & 1 & g}.
\end{displaymath}
  \begin{claim}\label{GGK12}
    For $w \in \{u,v\}$, $M_{B}[\{w,a,1,2,3,x_0, \ldots, x_k\}]$ is $3$\dash connected if and only if $A_{3w} \neq 0$. Furthermore, if $A_{3w} = 0$, then $(\{w,a,x_0,\ldots,x_k\},\{1,2,3\})$ is the unique $2$\dash separation of $M_{B}[\{w,a,1,2,3,x_0,\ldots,x_k\}]$.
  \end{claim}
  \begin{subproof}
    Let $Z := \{w,a,1,2,3,x_0,\ldots,x_k\}$. Clearly $r \neq 0$.
It is easy to see that if $\bip(A[Z\setminus\{2,3\}])$ contains a
split, then $k = 0$.
Claim~\ref{GGK9} implies that $r \not \in \{1,g\}$.
Now it follows from Propositions~\ref{GGK4.11} and~\ref{GGK4.12}
that $M_{B}[Z\setminus\{2,3\}]$ is $3$\dash connected. Hence $(\{w,a,x_0,\ldots,x_k\},\{1,2\})$ is the unique $2$\dash separation in $M_{B}[Z]- 3$. Now $A[\{a,1,2,3\}]$ is a twirl, so Proposition~\ref{GGK4.12} implies that $(\{3,w,a,x_0,\ldots,x_k\},\{1,2\})$ is not a $2$\dash separation in $M_{B}[Z]$.
Moreover $(\{w,a,x_0,\ldots,x_k\},\{1,2,3\})$ is a $2$\dash separation of $M_{B}[Z]$ if and only if $A_{3w} = 0$. If $A_{3w} \neq 0$ then $3$ is a blocking sequence for $(\{w,a,x_0,\ldots,x_k\},\{1,2\})$, and Proposition~\ref{GGK4.19} implies that $M_{B}[Z]$ is $3$\dash connected. If $A_{3w} = 0$ then it follows without difficulty from Proposition~\ref{GGK4.18} that $(\{w,a,x_0,\ldots,x_k\},\{1,2,3\})$ is the unique $2$\dash separation of $M_{B}[Z]$.
  \end{subproof}

  \begin{claim}\label{GGK13}
    We may assume $A_{3v} \neq 0$.
  \end{claim}
  \begin{subproof}
    Suppose $A_{3v} = A_{3u} = 0$ (if $A_{3u} \neq 0$ then we may swap $u$ and $v$). Then a pivot over $3x$ is allowable for all $x$ such that $A_{3x} \neq 0$. Claim~\ref{GGK12} implies that
    \begin{equation}\label{eqn3}
(\{v,a,x_0,\ldots,x_k\},\{1,2,3\})      
    \end{equation}
is the unique $2$\dash separation in $M_{B}[\{v,a,1,2,3,x_0,\ldots,x_k\}]$. If $k=0$ then the $2$\dash separation in~\eqref{eqn3} is not an induced separation of $M_{B}-u$, because $M \del u$ is stable, and $A[\{a,1,2,3\}]$ and $A[\{v,a,b,1\}]$ are twirls (since $r \notin \{0,1,g\}$). Now suppose that $k = 2$, and that the $2$\dash separation in~\eqref{eqn3} is induced in $M_{B}-u$.
Our choice of $u$ and $v$ implies that
$M_{B}-\{u,v\}$ contains a $3$\dash connected nonbinary minor of size at
least $|E|-4$.
It follows that $(E-\{u,1,2,3\},\{1,2,3\})$ must be a $2$\dash separation
of $M_{B}-u$, and that $M_{B}-\{u,v,1,3\}$ is $3$\dash connected and
nonbinary.
But since $A[\{a,1,2,3\}]$ is a twirl, we now have a contradiction to the
fact that $M\del\{u,v\}$ is stable.
Thus, in either case, the $2$\dash separation in~\eqref{eqn3} is not induced in $M_{B}-u$.
We let $v_1, \ldots, v_p$ be a blocking sequence, and we suppose that, subject
to~\ref{GGK1} and~\ref{GGK3}, we have chosen $u$, $v$, $B$, $a$, $b$, and
$C$ such that $p$ is as small as possible.

    First suppose $v_p$ labels a row. Then $(\{v,a,x_0,\ldots,x_k,v_p\},\{1,2,3\})$ is not a $2$\dash separation in $M_{B}[\{v,a,1,2,3,x_0,\ldots,x_k,v_p\}]$, so $A_{v_pw} \neq 0$ for some $w \in \{1,2\}$. By pivoting over $31$ or $32$ as needed, we may assume $A_{v_pw} \neq 0$ for all $w \in \{1,2\}$.
Then $(\{v,a,x_{0},\ldots, x_{k},v_{p}\},\{1,2,3\})$ is a split in
$\bip(A[\{v,a,1,2,3,x_0,\ldots,x_k,v_p\}])$.
Since it is not a $2$\dash separation, it follows without
difficulty from Proposition~\ref{GGK4.12}
that $A[\{a,v_{p},1,2\}]$ is a twirl.
If $p = 1$ then either $A_{v_pv} \neq 0$ or $A_{v_px_1} \neq 0$ (in the
case that $k=2$).
If $A_{v_{p}v}\neq 0$, then we can replace $3$ with $v_{p}$, and we are done.
Therefore we assume that $A_{v_{p}v}=0$ and that $A_{v_{p}x_{1}} \neq 0$.
But then $d(b,\{a,1,2,v_p\}) < d(b,\{a,1,2,3\})$, contradicting~\ref{GGK3}. Therefore $p>1$.
Now it follows from Proposition~\ref{GGK4.16}~\eqref{it:4.16.i} that $v_1,\ldots,v_{p-1}$ is a blocking sequence for the $2$\dash separation $(\{v,a,x_0,\ldots,x_k\},\{1,2,v_p\})$ of $M_{B}[\{v,a,1,2,v_p,x_0,\ldots,x_k\}]$.
This contradicts our assumption of minimality.

    Therefore we assume that $v_p$ labels a column. It follows that $A_{3v_p} \neq 0$ and, by pivoting on $A_{32}$ as necessary, $A_{av_p} \neq 0$. Lemma~\ref{GGK4.4} implies that $A[\{a,1,3,v_p\}]$ is a twirl
(we swap the labels of columns $1$ and $2$ as necessary). By pivoting over $13$ as necessary, we can assume that $A_{x_kv_p}\neq 0$. Now consider replacing $2$ by $v_p$. If $p>1$ then Proposition~\ref{GGK4.16}~\eqref{it:4.16.i} again implies that $v_1, \ldots, v_{p-1}$ is a blocking sequence for the $2$\dash separation $(\{v,a,x_0,\ldots,x_k\},\{1,3,v_p\})$ of $M_{B}[\{v,a,1,3,v_p,x_0,\ldots,x_k\}]$, contradicting our assumption of minimality. Therefore $p = 1$. Then $(\{v,a,x_0,\ldots,x_k\},\{1,2,3,v_p\})$ is a split in
$\bip(A[\{v,a,1,2,3,v_p,x_0,\ldots,x_k\}])$, but is not a $2$\dash separation of $M_{B}[\{v,a,1,2,3,v_p,x_0,\ldots,x_k\}]$.
Therefore Proposition~\ref{GGK4.12} implies that $A[\{a,1,x_k,v_p\}]$ is a twirl. But $d(b,\{a,x_k,1,v_p\})<d(b,\{a,1,2,3\})$, and again we have a contradiction
to~\ref{GGK3}.
This completes the proof of the claim.
  \end{subproof}

  \begin{claim}\label{GGK14}
    $A_{3u} = 0$.
  \end{claim}
  \begin{subproof}
    Suppose $A_{3u} \neq 0$. Let $Z := \{u,v,a,x_0,\ldots,x_k,1,2,3\}$. By Claim~\ref{GGK12}, $M_{B}[Z]-u$ and $M_{B}[Z]-v$ are both $3$\dash connected, and therefore stable. Furthermore, $M_{B}[Z]-\{u,v\}$ is certainly nonbinary.
By examining $\bip(A[Z]-\{u,v\})$, we see that $M_{B}[Z]-\{u,v\}$ is connected and stable, so Claim~\ref{GGK7} implies that $Z = E$. Since $\rank(M) > 4$ or $\corank(M) > 4$, this means that $k = 2$.
Now $b$ is in a series pair in $M_{B}-\{u,v\}$, and
$x_{1}$ is in a parallel pair in $M_{B}-\{u,v,b\}$.
We have chosen $u$ and $v$ so that $M_{B}-\{u,v\}$ has
a $3$\dash connected nonbinary minor of size at least
$|E|-4$, and this minor must be
$M_{B}-\{u,v,b,x_{1}\}$.
But $\{a,x_{2}\}$ is a series pair in this matroid, so we have a
contradiction.
  \end{subproof}

  Although the page count suggests otherwise, we are now entering the endgame of the proof: from now on we will deal only with the $2$\dash separation
$(\{u,a,x_0,\ldots,x_k\},\{1,2,3\})$ of
$M_{B}[\{u,a,1,2,3,x_0,\ldots,x_k\}]$.
That this is a $2$\dash separation follows from
Claims~\ref{GGK12} and~\ref{GGK14}.
Assume that it is induced in $M_{B}-v$.
If $k=0$ then this immediately leads to a contradiction, as
$M \del v$ is stable, and $A[\{a,1,2,3\}]$ and $A[\{u,a,b,1\}]$ are twirls.
Now suppose that $k = 2$.
There is a $3$\dash connected nonbinary minor of size at
least $|E|-4$ in $M_{B}-\{u,v\}$.
Therefore $(E-\{v,1,2,3\},\{1,2,3\})$ is a $2$\dash separation
of $M_{B}-v$, and $M_{B}-\{u,v,1,3\}$ is $3$\dash connected and
nonbinary.
Since $A[\{a,1,2,3\}]$ is a twirl, this contradicts the
fact that $M\del\{u,v\}$ is stable.
Thus, in either case, $(\{u,a,x_0,\ldots,x_k\},\{1,2,3\})$ is
not induced in $M_{B}-u$.
Therefore there exists a blocking sequence $v_1, \ldots, v_p$ in $M_{B}-v$.
Assume that, subject to~\ref{GGK1}, \ref{GGK3}, \ref{GGK13}, and
\ref{GGK14}, $B$, $a$, $b$, $C$, $x_1,\ldots, x_k$, and $v_1,\ldots, v_p$
have been chosen so that $p$ is as small as possible.

  \begin{claim}\label{GGK15}
    $p \neq 1$.
  \end{claim}
  \begin{subproof}
    Suppose $p = 1$, and let $Z := \{u,v,a,x_0,\ldots,x_k,1,2,3,v_1\}$. Claims~\ref{GGK12} and~\ref{GGK13} imply that $M_{B}[Z]-\{u,v_1\}$ is $3$\dash connected, so $M_{B}[Z]-u$ is stable. Claims~\ref{GGK12} and~\ref{GGK14} imply that
$(\{u,a,x_0,\ldots,x_k\},\{1,2,3\})$ is the unique $2$\dash separation in $M_{B}[Z]-\{v,v_1\}$. Since $v_1$ is a blocking sequence for this $2$\dash separation, Proposition~\ref{GGK4.19} states that $M_{B}[Z]-v$ is $3$\dash connected, and therefore stable.
It is easy to see that $M_{B}[Z]-\{u,v\}$ is connected and nonbinary.
Suppose that $M_{B}[Z]-\{u,v\}$ is not stable.
If $k=0$ then the only $2$\dash separation in $M_{B}[Z]-\{u,v,v_{1}\}$
is $(\{a,x_{0}\},\{1,2,3\})$.
If $k=2$, then $(\{1,2,3\},\{a,x_{0},x_{1},x_{2}\})$
is a $2$\dash separation.
Since $M_{B}[Z]-\{u,v\}$ is not stable, and $M_{B}[\{a,1,2,3\}]$
is nonbinary, it follows that
we can create a $2$\dash separation of $M_{B}[Z]-\{u,v\}$
by adding $v_{1}$ to the side of one of these separations
that does not contain $\{1,2,3\}$.
However, $u$ is spanned by $\{a,x_{0}\}$ (in the case that
$k=0$) or $\{a,x_{0},x_{2}\}$ (in the case that $k=2$) in $M_{B}[Z]-v$.
It follows that $M_{B}[Z]-v$ contains a $2$\dash separation,
contradicting our earlier conclusion that it is $3$\dash connected.
Therefore $M_{B}[Z]-\{u,v\}$ is stable.
Now Claim~\ref{GGK7} implies that $Z = E$.

    Suppose $k = 0$. Since $M$ has either rank or corank at least $5$, it follows that $v_1$ labels a column. Since $v_1$ is a blocking sequence, neither $(\{u,a,b,v_{1}\},\{1,2,3\})$ nor $(\{u,a,b\},\{v_{1},1,2,3\})$ is a $2$\dash separation of $M_{B}-v$. Now Lemma~\ref{lem:matrixconn} implies that $\matrank(A[\{u,v_1,3\}]) > 0$ and $\matrank(A[\{a,b,v_{1},1,2\}]) > 1$. Hence $A_{3v_1} \neq 0$ and one of $A_{av_1}$ and $A_{bv_1}$ is nonzero. Note that exactly one of these is nonzero, because otherwise $A[\{a,b,1,v_1\}]$ forms a twirl, and we have reduced to the case that $a$ and $b$ are contained in $C$. By swapping $a$ and $b$ if necessary, we assume that $A_{av_1} \neq 0$. Now we consider the ternary matrix $\phi(A)$. Recall that $M = M[I|\phi(A)]$. Up to scaling we may assume
    \begin{displaymath}
      \phi(A) = \kbordermatrix{ & 1 & 2 & v_1 & u & v\\
                          3  & q & 1 & t   & 0 & s\\
                          a  & 1 & 1 & 1   & 1 & 1\\
                          b  & r & r & 0   & 1 & g},
    \end{displaymath}
where $g$, $q$, $r$, $s$, and $t$ are all nonzero.
    Since $M_{B}[\{a,1,2,3\}] = N_{B}[\{a,1,2,3\}] \cong U_{2,4}$,
it follows that $q = -1$. Claim~\ref{cl:distsetindependent} implies that
$B' = (B\setminus\{a,b\})\cup \{u,v\}$ is dependent in $M$, so $g = 1$. Now Claim~\ref{GGK9} implies that $r = -1$. By scaling row $3$ and swapping columns $1$, $2$ as necessary, we may assume $t = 1$. This leaves us to consider two choices for $s$. If $s = 1$ then $M\del 2 \cong F_{7}^{-}$.
But this contradicts our conclusion that $M$ is
$\GF(4)$\dash representable.
Therefore we assume that $s=-1$.
In this case $M \cong \agde$, which we assumed was not so.

    Therefore $k = 2$. Here we have to distinguish two cases. First, suppose $v_1$ labels a column. Since $v_1$ is a blocking sequence, we can argue as before, and deduce that $A_{3v_1} \neq 0$ while $A_{wv_1} \neq 0$ for at least one $w \in \{a,b,x_2\}$. Since $A_{3v_{1}}\neq 0$ and $d(b,C)=3$, it follows that $A_{bv_1} = 0$. As both $A[\{a,1,2,3\}]$ and $A[\{x_{2},1,2,3\}]$ are twirls, Proposition~\ref{GGK4.4} implies that one of $A[\{a,x_2,1,3,v_1\}]$ or $A[\{a,x_2,2,3,v_1\}]$ contains a twirl. By swapping $1$ and $2$ if necessary, we can assume that $A[\{a,x_2,2,3,v_1\}]$ contains a twirl. Claim~\ref{GGK12} implies that $M_{B}-\{v,v_{1},1\}$ has a unique $2$\dash separation, namely $(\{u,a,b,x_1,x_2\},\{2,3\})$.
It is easy to see that $(\{u,a,b,x_{1},x_{2},v_{1}\},\{2,3\})$ is not a $2$\dash separation in $M_{B}-\{v,1\}$. If $(\{u,a,b,x_{1},x_{2}\},\{v_{1},2,3\})$ is a $2$\dash separation of $M_{B}-\{v,1\}$, then $A_{av_{1}}$ and $A_{x_{2}v_{1}}$ must be nonzero, and $A[\{a,x_{2},v_{1},2]$ must have determinant zero. But this implies that $(\{u,a,b,x_{1},x_{2}\},\{v_{1},1,2,3\})$ is a $2$\dash separation of $M_{B}-v$, a contradiction. Therefore $v_{1}$ is a blocking sequence in $M_{B}-\{v,1\}$, so Proposition~\ref{GGK4.19} implies that $M_{B}-\{v,1\}$ is $3$\dash connected. Hence $M_{B}-v$, $M_{B}-1$, and $M_{B}-\{v,1\}$ are all stable, and $M_{B}-\{v,1\}$ is $3$\dash connected and nonbinary. Therefore $v$, $1$ is a deletion pair, and furthermore, $M_{B}-\{v,1\}$ certainly contains a $3$\dash connected nonbinary minor on at least $|E|-4$ elements.
Since $M\del \{v,1\}$ is $3$\dash connected, and $b$ is a degree-one vertex of
$\bip(A-\{u,v\})$, we now have a contradiction to~\ref{GGK1}.

    Next we suppose that $v_1$ labels a row. Suppose that $(\{a,x_{0},x_{1},x_{2}\},\{v_{1},1,2,3\})$ is a $2$\dash separation of $M_{B}-\{u,v\}$. Then $M_{B}-\{u,v\}$ cannot contain a $3$\dash connected minor of size at least $|E|-4$, which contradicts our choice of $u$ and $v$.
Therefore $(\{a,x_{0},x_{1},x_{2}\},\{v_{1},1,2,3\})$ is not a $2$\dash separation, so Lemma~\ref{lem:matrixconn} implies that $A_{v_1x_1} \neq 0$.
Similarly, $(\{u,a,x_{0},x_{2},x_{3},v_{1}\},\{1,2,3\})$ is not a $2$\dash separation in $M_{B}-v$, so $\matrank(A[\{a,v_1,1,2\}]) > 1$. It cannot be the case that $A[\{a,v_1,1,2\}]$ is a twirl, since $2=d(b,\{a,v_1,1,2\}) <d(b,C)=3$. Hence exactly one of $A_{v_11}$ and $A_{v_12}$ is nonzero; by relabeling as necessary we assume $A_{v_12} = 0$.
    \begin{sclaim}
      $M_{B}-\{u,v,3\}$ is binary.
    \end{sclaim}
    \begin{subproof}
By examining $\bip(A[\{a,b,1,2,x_{1},x_{2}])$, we see that $(\{a,1,2,x_2,v_1\},\{b,x_1\})$ is the unique $2$\dash separation in $M_{B}-\{u,v,3\}$. Therefore $M_{B}-\{u,v,3\}$ is stable. By inspection, $(\{a,1,2,x_{2},v_{1}\},\{b,x_{1}\})$ is uncrossed, and $u$ and $v$ are blocking sequences. Now by using Proposition~\ref{GGK4.17}, we can see that $M_{B}-\{u,3\}$ and $M_{B}-\{v,3\}$ must be stable.
Certainly $M_{B}-\{u,v,3\}$ is connected.
If it were nonbinary, then Claim~\ref{GGK7} would imply that $E\setminus 3=E$.
Therefore $M_{B}-\{u,v,3\}$ is binary, as desired.
    \end{subproof}

    \begin{sclaim}
      $A[\{v,a,2,3\}]$ is a twirl.
    \end{sclaim}
    \begin{subproof}
Note that Proposition~\ref{GGK4.4} implies either
$A[\{a,v,1,3\}]$ or $A[\{a,v,2,3\}]$ is a twirl.
Let us assume that the claim fails, so that $A[\{v,a,1,3\}]$
is a twirl.
Consider $\bip(A-\{u,x_{2}\})$.
There are two splits in this graph:
$(\{b,v,v_{1},x_{1},2\},\{a,3\})$ and
$(\{b,v,v_{1},x_{1}\},\{a,2,3\})$.
Proposition~\ref{GGK4.12} implies that
neither of these is a $2$\dash separation, so
$M_{B}-\{u,x_{2}\}$ is $3$\dash connected.

By repeatedly cosimplifying and simplifying, we reduce
$M_{B}-\{u,v,x_2\}$ to a whirl.
Therefore $M_{B}-\{u,v,x_2\}$ is nonbinary and stable.
It is easy to see that it is connected.
There are no splits in $\bip(A-\{v,x_{2}\})$, so
$M_{B}-\{v,x_2\}$ is $3$\dash connected.
Now Claim~\ref{GGK7} implies $E\setminus x_{2}=E$, and
we have a contradiction.
    \end{subproof}
Since $M_{B}-\{u,v,3\}$ is binary, $A[\{x_{1},x_{2},v_{1},1\}]$ is
not a twirl.
Therefore $A_{v_{1}1} = A_{v_{1}x_{1}}$.
By scaling row $v_{1}$, we can assume that $A_{v_{1}1}=A_{v_{1}x_{1}}=1$.
Now
    \begin{displaymath}
      A = \kbordermatrix{    & 1 & 2 & x_1 & u & v\\
                           3 & q & 1 & 0   & 0 & s\\
                           a & 1 & 1 & 0   & 1 & 1\\
                         x_2 & 1 & 1 & 1   & * & *\\
                           b & 0 & 0 & r   & 1 & g\\
                         v_1 & 1 & 0 & 1   & * & *},
\quad 
      A^{a2} = \kbordermatrix{    & 1 & a & x_1 & u & v\\
                                3 & q-1 & \rbox{-1} & 0   & \rbox{-1} & s-1\\
                                2 & 1  & \rbox{1} & 0   & \rbox{1} & 1\\
                              x_2 & 0 & \rbox{-1} & 1   & * & *\\
                                b & 0  & \rbox{0} & r   & \rbox{1} & g\\
                              v_1 & 1  & \rbox{0} & 1   & * & *}
    \end{displaymath}
The fact that $A[\{v,a,2,3\}]$ is a twirl means that $s\neq 1$.
Since $A-\{u,v\}$ is a near-unimodular matrix,
we see that $q$ is a fundamental element of $\nreg$.
We write $B'$ for $B\symdiff \{a,2\}$ and $A'$ for $A^{a2}$.
    \begin{sclaim}
      We may assume that one of $A'[\{u,a,x_{2},3\}]$ and $A'[\{u,v_{1},1,3\}]$ is a twirl.
    \end{sclaim}
    \begin{subproof}
Assume that neither $A'[\{u,a,x_{2},3\}]$ nor $A'[\{u,v_{1},1,3\}]$
is a twirl.
The fact that $A'[\{u,a,x_{2},3\}]$ is not a twirl means that
$A'_{x_{2}u} \in \{0,-1\}$.
Similarly, since $A'[\{u,v_{1},1,3\}]$ is not a twirl we deduce that
$A'_{v_{1}u} \in \{0,1/(1-q)\}$.
Now we pivot on $bx_{1}$ and swap the labels on $b$ and $x_{1}$.
If $A'_{x_{2}u}$ is no longer $0$ or $-1$, then
$A'[\{u,a,x_{2},3\}]$ is a twirl, and we are done.
Therefore we assume that after this pivot, $A'_{x_{2}u}$ is still
either $0$ or $-1$, so $r \in \{1,-1\}$.
Similarly, we assume that after the pivot, $A'_{v_{1}u}$ is still
either $0$ or $1/(1-q)$.
This means that $r$ is either $q-1$ or $1-q$.
We deduce that $q-1$ is equal to either $1$ or $-1$.
But $q$ is an element of $\nreg$, and is
therefore not equal to $2$.
Thus $q=0$, which contradicts the fact that $A[\{a,1,2,3\}]$ is a twirl.
This completes the proof of the claim.
    \end{subproof}

Now we let $C'$ be either $\{u,a,x_{2},3\}$ or $\{u,v_{1},1,3\}$,
so that $A'[C']$ is a twirl.

    \begin{sclaim}\label{cl:contractpair}
      $M \con b = N \con b$ and $M \con 2 = N \con 2$.
Moreover $b$, $2$ is a deletion pair of $M^*$, $M_{B'}-\{b,2\}$ contains a $3$\dash connected nonbinary
minor of size at least $|E|-4$, 
and $|\co(M^*\del \{b,2\})| \geq |\co(M\del \{u,v\})|$.
    \end{sclaim}
    \begin{subproof}
Note that $A'[\{a,1,2,3\}]$ is a twirl by Proposition~\ref{GGK4.5}.
Therefore $M_{B'}-\{b,u,v\}$ is nonbinary.
By examining $\bip(A'-\{b,u,v\})$ and applying Propositions~\ref{GGK4.11}
and~\ref{GGK4.12} we see that $M_{B'}-\{b,u,v\}$ is $3$\dash connected.
Therefore $M_{B'}-\{b,u\}$ and $M_{B'}-\{b,v\}$ are both stable.
If $M\con b \neq N \con b$, then Claim~\ref{GGK7} implies that
$E\setminus b = E$.
Thus $M \con b = N \con b$.

By examining $\bip(A-\{u,v,2\})$, we see that $(\{v_1,2,3,a,x_2\},\{b,x_1\})$
is the only $2$\dash separation of $M_{B'}-\{u,v,2\}$.
Moreover, since $s-1 \neq 0$, both $u$ and $v$ are length-one blocking
sequences for this $2$\dash separation.
It now follows from Proposition~\ref{GGK4.19} that $M_{B'}-\{u,2\}$ and
$M_{B'}-\{v,2\}$ are both $3$\dash connected.
Therefore $M_{B'}-\{u,v,2,\}$, $M_{B'}-\{u,2\}$, and $M_{B'}-\{v,2\}$
are all stable.
Moreover $M_{B'}-\{u,v,2\}$ is connected and nonbinary.
Now it follows from Claim~\ref{GGK7} that $M\con 2 = N \con 2$.

As $A'[C']$ is a twirl, it follows without difficulty from
Propositions~\ref{GGK4.11} and~\ref{GGK4.12} that
$M_{B'}-\{b,v,2\}$ is $3$\dash connected.
Hence $M\con\{b,2\}$ is stable.
It is certainly nonbinary and connected.
We have noted that $M_{B'}-\{u,2\}$ is $3$\dash connected, so
$M\con 2$ is stable.
Note that $A'[\{u,1,2,3\}]$ is a twirl, as $q\neq 0$.
Now it follows from Proposition~\ref{GGK4.12} that
$M_{B'}-\{b,v\}$ is $3$\dash connected.
Thus $M\con b$ is stable.
Certainly $\{b,2\}$ is independent, so $b$, $2$ is
a deletion pair of $M^{*}$.

Since $M_{B'}-\{b,v,2\}$ is $3$\dash connected, it
follows that $M_{B'}-\{b,2\}$ contains a $3$\dash connected
nonbinary minor on at least $|E|-3$ elements.
Either $M_{B'}-\{b,2\}$ is $3$\dash connected, or
it contains a single parallel pair, and this
pair contains $v$.
In either case $|\co(M^{*}\del\{b,2\})| \geq |E|-3$.
As $b$ is in a series pair in $M \del \{u,v\}$, we see that
$|\co(M\del\{u,v\})| \leq |E|-3$, so we are done.
    \end{subproof}

By Lemma~\ref{GGK2.2} and Claim~\ref{cl:contractpair}, $A'$ is the
unique matrix over \qalpha\ such that $M\con b = M[I| A'-b]$ and
$M\con 2 = M[I| A'-2]$.
Now $\{2,b,u,v\}$ distinguishes $M$ from $N = M[I| A']$, so if we
replace $M$ by $M^*$, $u$ and $v$ with $2$ and $b$, replace
$a$ and $b$ with $u$ and $v$, $B$ with $B'$, and
$C$ with $C'$, then we have not violated~\ref{GGK1}.
However in $\bip(A'-\{b,2\})$, the distance between
$v$ and $C'$ is $1$, which is less than $d(b,C)$.
Thus we have a contradiction to~\ref{GGK3}, and this completes the
proof of Claim~\ref{GGK15}.
  \end{subproof}

  \begin{claim}\label{GGK16}
    $v_p$ labels a row.
  \end{claim}
  \begin{subproof}
    Suppose $v_p$ labels a column. Since $(\{u,a,x_0,\ldots,x_k,v_p\},\{1,2,3\})$ is not a $2$\dash separation in $M_{B}[\{u,a,x_0,\ldots,x_k,1,2,3,v_p\}]$, it follows that $A_{3v_p} \neq 0$. Claim~\ref{GGK15} says that $p > 1$, so the definition of blocking sequences implies that $(\{u,a,x_0,\ldots,x_k\},\{1,2,3,v_p\})$ is a $2$\dash separation. Then $\matrank(A[\{a,x_0,1,2,v_p\}]) = 1$ (in the case that $k=0$), or $\matrank(A[\{a,x_0,x_2,1,2,v_p\}]) = 1$ (if $k=2$).
It follows from this that either $A_{av_p} = A_{x_kv_p} = 0$, or both
$A_{av_{p}}$ and $A_{x_{k}v_{p}}$ are nonzero.
Moreover, if $k=2$, then $A_{bv_{p}}=0$.
Suppose that $A_{av_{p}}$ and $A_{x_{k}v_{p}}$ are nonzero.
Lemma~\ref{GGK4.4} and Claim~\ref{GGK13} imply that one of
$A[\{v_p,a,1,3\}]$ and $A[\{v_p,a,2,3\}]$ is a twirl. By swapping the labels of columns $1$ and $2$ as needed, assume $A[\{v_p,a,1,3\}]$ is a twirl.
By taking $Y'=\{1,3\}$ and applying Proposition~\ref{GGK4.16}~\eqref{it:4.16.i} we see that $v_1, \ldots, v_{p-1}$ is a blocking sequence for $(\{u,a,x_0,\ldots,x_k\},\{v_p,1,3\})$ in $M_{B}[\{u,a,x_{0},\ldots, x_{k},1,3,v_{p}]$.
Now we can replace $2$ with $v_{p}$, and we obtain a contradiction to
the minimality of $p$.

It follows that $A_{av_p} = A_{x_kv_p} = 0$. Since $A_{bv_{p}}=0$ if $k=2$,
this means that $3$ is the only neighbor of $v_p$ in $\bip(A[\{u,v,a,x_0,\ldots,x_k,1,2,3,v_p\}])$, so $\{3,v_{p}\}$ is a parallel pair in $M_{B}[\{u,v,a,x_0,\ldots,x_k,1,2,3,v_p\}]$.
Therefore $M_{B\symdiff \{3,v_p\}}[\{u,v,a,x_0,\ldots,x_k,1,2,v_p\}]$
is isomorphic to $M_B[\{u,v,a,x_0,\ldots,x_k,1,2,3\}]$.
It is very easy to verify that
\begin{displaymath}
\lambda_{B}(\{u,a,x_0,\ldots,x_k,3\},\{1,2,3\}) = \lambda_{B}(\{u,a,x_0,\ldots,x_k,3\},\{1,2\})=2.
\end{displaymath}
Proposition~\ref{GGK4.16}~\eqref{it:4.16.ii} implies that $v_1,\ldots,v_{p-1}$ is a blocking sequence for $(\{u,a,x_0,\ldots,x_k\},\{1,2,v_p\})$ of $M_{B\symdiff \{3,v_p\}}[\{u,a,x_0,\ldots,x_k,1,2,v_p\}]$.
By replacing $3$ with $v_{p}$ we obtain a contradiction to the
minimality of $p$.
  \end{subproof}

  \begin{claim}\label{GGK17}
    $p \neq 2$.
  \end{claim}
  \begin{subproof}
    Suppose $p = 2$. Then $v_1$ labels a column and $v_2$ labels a row by Claim~\ref{GGK16}. As $(\{u,a,x_{0},\ldots, x_{2},v_{1}\},\{1,2,3\})$ is
a $2$\dash separation of $M_{B}[\{u,a,x_{0},\ldots, x_{2},v_{1}\}]$
it follows that $A_{3v_1} = 0$.
On the other hand, $A[\{a,x_{0},x_{k}\},\{v_{1}\}]$ is not the zero matrix.
Suppose $A_{zv_1} \neq 0$ for exactly one $z \in \{x_0,x_k,a\}$.
Then a pivot over $zv_1$ is allowable, and $v_{1}$ is parallel to
$z$ in $M_{B}[\{u,v,a,x_0,\ldots,x_k,1,2,3,v_{1}\}]$.
Therefore
\begin{multline*}
M_{B\symdiff \{z,v_1\}}[\{u,v,a,x_0,\ldots,x_k,1,2,3\}\symdiff\{z,v_1\}]\\
\cong M_{B}[\{u,v,a,x_0,\ldots,x_k,1,2,3\}].
\end{multline*}
Moreover $\lambda_{B}(\{u,a,x_0,\ldots,x_k\},\{1,2,3,z\}) = 2$,
so Proposition~\ref{GGK4.16}~\eqref{it:4.16.ii}, and symmetry,
implies that $v_2$ is a blocking sequence for
$(\{u,a,x_0,\ldots,x_k\}\symdiff \{z,v_1\},\{1,2,3\})$ in
$M_B[\{u,a,x_0,\ldots,x_k,1,2,3\}\symdiff\{z,v_1\}]$.
Now we can replace $z$ with $v_{1}$, and derive a
contradiction to the minimality of $p$.

    Hence $A_{zv_1} \neq 0$ for at least two elements $z \in \{a,x_0,x_2\}$. Suppose $k = 2$ and $A_{x_0v_1}\neq 0$. Since $d(b,C) = 3$ we have that $A_{av_1} = 0$ and hence $A_{x_2v_1} \neq 0$.
We consider replacing $x_1$ by $v_1$.
By using symmetry and Proposition~\ref{GGK4.16}~\eqref{it:4.16.i},
with $Y' = \{u,a,x_{0},x_{2}\}$, we see that $v_2$ is a blocking sequence for $(\{u,a,x_0,v_1,x_2\},\{1,2,3\})$ in $M_B[\{u,a,x_0,v_1,x_2,1,2,3\}]$.
But this leads to a contradiction, as the minimality of $p$
is violated. It follows that, for $k = 0$ and for $k = 2$,
both $A_{av_1}$ and $A_{x_kv_1}$ are nonzero. Since
$(\{u,a,x_{0},\ldots, x_{2}\},\{v_{1},1,2,3\})$ is not a $2$\dash separation,
it follows that $\matrank(A[\{a,x_k\},\{v_{1},1,2\}]) > 1$, and therefore
$A[\{a,x_k,1,v_1\}]$ is a twirl.
But $d(b,\{a,x_k,1,v_1\}) < d(b,C)$, contradicting~\ref{GGK3}.
  \end{subproof}

  Define $Z := \{u,v,a,x_0,\ldots,x_k,1,2,3,v_{p-1},v_p\}$. By Claim~\ref{GGK16}, $v_p$ labels a row, and hence $v_{p-1}$ labels a column. From the definition of blocking sequence we find that both $(\{u,a,x_0,\ldots,x_k,v_{p-1}\},\{1,2,3\})$ and $(\{u,a,x_0,\ldots,x_k\},\{1,2,3,v_{p-1}\})$ are $2$\dash separations in $M_B[Z\setminus \{v,v_p\}]$. It follows from Lemma~\ref{lem:matrixconn} that $A_{3v_{p-1}} = 0$.
As $(\{u,a,x_0,\ldots,x_k,\},\{v_{p},1,2,3\})$ is
a $2$\dash separation, but
$(\{u,a,x_0,\ldots,x_k,v_{p-1}\},\{v_{p},1,2,3\})$ is not,
it follows that $A_{v_pv_{p-1}} \neq 0$.
Furthermore, either $A_{av_{p-1}} = A_{x_kv_{p-1}} = 0$ or both $A_{av_{p-1}}$
and $A_{x_{k}v_{p-1}}$ are nonzero.
Suppose $A_{av_{p-1}} = A_{x_kv_{p-1}} = 0$.
As $(\{u,a,x_0,\ldots,x_k,\},\{v_{p-1},1,2,3\})$ is a
$2$\dash separation, it follows that $v_{p}$ is the only
neighbor of $v_{p-1}$ in $\bip(A[Z \setminus v_{p}])$.
Thus $M_{B\symdiff \{v_p,v_{p-1}\}}[Z]-\{v_{p-1},v_p\}$ is isomorphic to $M_B[Z]-\{v_{p-1},v_p\}$. Proposition~\ref{GGK4.16}~\eqref{it:4.16.iii} implies that
$v_1,\ldots,v_{p-2}$ is a blocking sequence for the $2$\dash separation $(\{u,a,x_0,\ldots,x_k\},\{1,2,3\})$ in $M_{B\symdiff\{v_{p-1},v_p\}}$, and we have
contradicted the minimality of $p$.
Therefore $A_{av_{p-1}}$ and $A_{x_kv_{p-1}}$ are both nonzero.

  At least one of $A_{v_p1}$ and $A_{v_p2}$ needs to be nonzero. Assume, exchanging $1$ and $2$ if necessary, that $A_{v_p1} \neq 0$. We deduce that, up to scaling,
$A[Z]$ is equal to one of the following matrices:
  \begin{displaymath}
    \kbordermatrix{ & 1 & 2 & v_{p-1} & u & v\\
                  3 & q & 1 &    0    & 0 & s\\
                v_{p} & 1 & * &    t    & 0 & *\\
                  a & 1 & 1 &    1    & 1 & 1\\
            x_0 = b & r & r &    r    & 1 & g}, \quad
    \kbordermatrix{ & 1 & 2 & v_{p-1} & x_1 & u & v\\
                  3 & q & 1 &    0    & 0   & 0 & s\\
                v_{p} & 1 & * &    t    & 0   & 0 & *\\
                  a & 1 & 1 &    1    & 0   & 1 & 1\\
                x_2 & 1 & 1 &    1    & 1   & * & *\\
            x_0 = b & 0 & 0 &    0    & r   & 1 & g}.
  \end{displaymath}
  \begin{claim}\label{GGK18}
    $A[\{a,v_{p-1},v_p,1\}]$ is not a twirl.
  \end{claim}
  \begin{subproof}
    Suppose  $A[\{a,v_{p-1},v_p,1\}]$ is a twirl.
By applying Proposition~\ref{GGK4.16}~\eqref{it:4.16.i} twice,
we see that $v_1,\ldots,v_{p-2}$ is a blocking sequence for the
$2$\dash separation $(\{u,a,x_0,\ldots,x_k\},\{1,v_{p-1},v_p\})$
of $M_{B}[Z\setminus\{v,2,3\}]$.
If $A_{v_pv} \neq 0$ then we can replace $3$ with $v_{p}$ and
$2$ with $v_{p-1}$, and derive a contradiction to the  minimality of $p$.
Therefore $A_{v_pv} = 0$.
Now a pivot over $v_p1$ is allowable, and
by scaling, we can assume that $A^{v_p1}[Z]- 2$ is one of the
following matrices:
    \begin{displaymath}
      \kbordermatrix{ & v_p & v_{p-1} & u & v\\
                    3 & q   &    q    & 0 & s'\\
                    1 & 1   &    1    & 0 & *\\
                    a & 1   & \beta & 1 & 1\\
              x_0 = b & r   & r\beta& 1 & g'}, \quad
      \kbordermatrix{ & v_p & v_{p-1} & x_1 & u & v\\
                    3 & q   &    q    & 0   & 0 & s'\\
                    1 & 1   &    1    & 0   & 0 & *\\
                    a & 1   & \beta & 0   & 1 & 1\\
                  x_2 & 1   & \beta & 1   & * & *\\
              x_0 = b & 0   &    0    & r   & 1 & g}.
    \end{displaymath}
Here, $\beta = (t-1)/t$, $s' = s-qA_{v_{p}v}$, and $g'=g-qA_{v_{p}v}$.
Our assumption that $A[\{a,v_{p-1},v_p,1\}]$ is a twirl implies
$t$ is a fundamental element of $\nreg$ other than zero or
one, so $\beta$ is defined, and is equal to neither $0$ nor $1$.
Hence $A^{v_p1}[\{a,v_{p-1},v_{p},3\}]$ is a twirl.

Proposition~\ref{GGK4.15}~\eqref{it:4.15.i} and~\eqref{it:4.15.ii}
implies that $v_1, \ldots, v_{p-1}$ is a blocking sequence for the $2$\dash separation $(\{u,a,x_0,\ldots,x_k\},\{v_p,1,2,3\})$ in $M_{B\symdiff\{1,v_p\}}[Z\setminus\{v,v_{p-1}\}]$. 
Proposition~\ref{GGK4.16}~\eqref{it:4.16.i} implies that $v_1,\ldots, v_{p-2}$ is a blocking sequence for $(\{u,a,x_0,\ldots, x_k\},\{v_{p-1},v_p,3\})$ in $M_{B\symdiff\{v_p,1\}}[Z]- \{v,1,2\}$. It follows that we can replace $B$ by $B\symdiff\{v_p,1\}$ and $C$ by $\{a,v_p,v_{p-1},3\}$, which contradicts the minimality of $p$.
  \end{subproof}

  \begin{claim}\label{GGK19}
    $A_{v_p2} = 0$.
  \end{claim}
  \begin{subproof}
    Suppose $A_{v_p2} \neq 0$. As $(\{u,a,x_0,\ldots,x_k,v_p\},\{1,2,3\})$ is
not a $2$\dash separation of $M_{B}[Z]-\{v,v_{p-1}\}$
it follows that $A[\{a,v_{p},1,2\}]$ must be a twirl.
Hence either $A[\{a,v_{p-1},v_{p},1\}]$ or $A[\{a,v_{p-1},v_{p},2\}]$
is a twirl, by Lemma~\ref{GGK4.4}.
But, possibly after exchanging $1$ and $2$, this contradicts Claim~\ref{GGK18}.
  \end{subproof}

  \begin{claim}\label{GGK20}
    $A_{v_pv} \neq 0$.
  \end{claim}
  \begin{subproof}
    Suppose $A_{v_pv} = 0$. Then $v_{p}1$ is an allowable pivot, and
$v_{p}$ is adjacent only to $1$ in
$\bip(A[Z]-v_{p-1})$.
Thus $M_{B\symdiff \{v_p,1\}}[Z]-\{v_{p-1},1\}$ is isomorphic to
$M_B[Z]-\{v_{p-1},v_p\}$.
Furthermore
\begin{displaymath}
\lambda_{B}(\{u,a,x_0,\ldots,x_k,1\},\{1,2,3\})= \lambda_{B}(\{u,a,x_0,\ldots,x_k,1\},\{2,3\}) = 2,
\end{displaymath}
so Proposition~\ref{GGK4.16}~\eqref{it:4.16.ii} then implies that $v_1,\ldots,v_{p-1}$ is a blocking sequence for $(\{u,a,x_0,\ldots,x_k\},\{v_p,2,3\})$ of $M_{B\symdiff \{1,v_p\}}[Z\setminus\{v_{p-1},1\}]$. This contradicts the minimality of $p$.
  \end{subproof}

Let $\psi$ be an automorphism of the near-regular partial field.
Then both $A-u$ and $\psi(A)-u$ represent $M \del u$ over
$\nreg$.
Similarly, $A-v$ and $\psi(A)-v$ represent $M \del v$ over
$\nreg$.
Obviously $\psi(A)-u$ and $\psi(A)-v$ are both near-unimodular.
Thus $\psi(A)$ satisfies the conditions of Lemma~\ref{GGK2.2},
so Lemma~\ref{lem:ternaryhom} implies that
$M=M[I|\phi(\psi(A))]$, where $\phi$ is the homomorphism from
$\nreg$ to $\GF(3)$.

Consider $A[\{a,3\},\{1,2\}]$.
It is a submatrix of the near-unimodular matrix $A-u$, and
its determinant is $q-1$.
We deduce that $q$ is a fundamental element of $\nreg$
other than $0$ and $1$.
By Proposition~\ref{prop16}, and the discussion in the
previous paragraph, we can assume that $q=\alpha$.

Since $B'=B\symdiff\{a,b,u,v\}$ is a dependent
set in $M=M[I|\phi(A)]$, the determinant of
$\phi(A)[\{a,b,u,v\}]$ must be zero evaluated over $\GF(3)$.
It follows that $\phi(g)=1$.
Claim~\ref{GGK9} implies that
$B \symdiff \{a,b,u,1\}$ is independent in
$N \del v = M \del v$.
Hence $\phi(A)[\{a,b,u,1\}]$ has a nonzero determinant
over $\GF(3)$.
Since $r \neq 0$, and the only element of $\nreg$ taken to
zero by $\phi$ is zero itself, it follows that
$\phi(r)=-1$.
It follows from Claim~\ref{GGK18}
that $\phi(A)[\{a,v_{p-1},v_{p},1\}]$ has zero determinant, so
$\phi(t)=1$.

Suppose $k = 0$.
By the preceding discussion we see that
    \begin{displaymath}
      \phi(A)[Z] = \kbordermatrix{ & 1 & 2 & v_{p-1} & u & v\\
                  3 & \rbox{-1} & \rbox{1} & \rbox{0} & \rbox{0} & s\\
                v_p & \rbox{1} & \rbox{0} & \rbox{1} & \rbox{0} & w\\
                  a & \rbox{1} & \rbox{1} & \rbox{1} & \rbox{1} & 1\\
            x_0 = b & \rbox{-1} & \rbox{-1} & \rbox{-1} & \rbox{1} & 1},
    \end{displaymath}
where $s$ and $w$ are nonzero.
By scaling the column labeled with $v$ and, if necessary,
scaling the column labeled by $u$, the rows labeled
by $a$ and $b$, and swapping the labels on the last two
rows, we can assume that $s=1$.
Thus there are two cases to consider, according to whether
$w$ is equal to $1$ or $-1$.

If $w=1$ then
$M_{B\symdiff\{v,3\}}[Z]-\{v,3\} \cong F_{7}^{-}$, which contradicts the
fact that $M$ is $\GF(4)$\dash representable.
Similarly, in the case that $w=-1$.
then it is easy to check that
$M_{B\symdiff\{a,u\}}[Z]-\{a,u\} \cong F_{7}^{-}$.

  Now we assume $k=2$.
Let $Z'=Z\setminus\{x_{1},x_{2}\}$, so that
  \begin{displaymath}
    \phi(A[Z']) =   \kbordermatrix{ & 1 & 2 & v_{p-1} & u & v\\
                  3 & \rbox{-1} & 1 &    0    &  0 & s\\
                v_p & \rbox{1} & 0 &    1       & 0 & w\\
                  a & \rbox{1} & 1 &    1       & 1 & 1\\
              x_0 = b & \rbox{0} & 0 &    0       & 1 & 1}.
  \end{displaymath}
We consider four cases.
If $(s,w)=(1,1)$ then
$M_{B\symdiff\{v,3\}}[Z']-\{v,1\} \cong F_{7}^{-}$.
If $(s,w)$ is equal to $(1,-1)$ or $(-1,1)$ then
$M_{B\symdiff\{a,u\}}[Z']-\{a,u\} \cong F_{7}^{-}$.
Finally, if $(s,w)=(-1,-1)$, then
$M_{B\symdiff\{3,v\}}[Z']-v \cong \agde$.
Thus $M$ has a minor isomorphic to \agde.
But we will show in Proposition~\ref{AG23} that
\agde\ is an excluded minor for the class of
near-regular matroids.
Thus $M \cong \agde$, which contradicts our assumption, and
completes the proof of Theorem~\ref{thm:finitesize}.
\end{proof}

\section{Conclusion}\label{sec:conclusion}

In this section we complete the proof of the
excluded-minor characterization.
We start by describing in detail the matroids 
listed in Theorems~\ref{thm4} and~\ref{thm1}, and proving that
they are indeed excluded minors for near-regularity.
Theorem~\ref{thm:finitesize} means that to prove this
list is complete, we need only perform a finite case-analysis.
That analysis is carried out in the second half of the section.

\subsection{The excluded minors}
\label{excludedminors}
The next result follows easily
from Proposition~6.5.2 of~\cite{oxley}.

\begin{prop}
\label{U25U35}
Both $U_{2,5}$ and $U_{3,5}$ are excluded minors for
the class of near-regular matroids.
\end{prop}

Recall that $F_{7}$, the Fano plane, and $F_{7}^{-}$,
the non-Fano matroid, are the rank\dash $3$ matroids shown in
Figure~\ref{figFano}.

\begin{figure}[htb]

\centering

\begin{pspicture}(-3.86,0)(3.86,2.2)

\SpecialCoor

\rput(-3,1.25){
\pscircle*(0,0){0.06cm}
\pscircle*(1;90){0.06cm}
\pscircle*(1;210){0.06cm}
\pscircle*(1;330){0.06cm}
\pscircle*(0.5;150){0.06cm}
\pscircle*(0.5;270){0.06cm}
\pscircle*(0.5;30){0.06cm}

\pspolygon(1;90)(1;210)(1;330)
\psline(1;90)(0.5;270)
\psline(1;210)(0.5;30)
\psline(1;330)(0.5;150)

\pscurve(0.5;150)(0.5;270)(0.5;30)
}

\rput[B](-3,0){$F_{7}$}

\rput(3,1.25){
\pscircle*(0,0){0.06cm}
\pscircle*(1;90){0.06cm}
\pscircle*(1;210){0.06cm}
\pscircle*(1;330){0.06cm}
\pscircle*(0.5;150){0.06cm}
\pscircle*(0.5;270){0.06cm}
\pscircle*(0.5;30){0.06cm}

\pspolygon(1;90)(1;210)(1;330)
\psline(1;90)(0.5;270)
\psline(1;210)(0.5;30)
\psline(1;330)(0.5;150)
}

\rput[B](3,0){$F_{7}^{-}$}

\end{pspicture}

\caption{The Fano plane, and the non-Fano matroid.}

\label{figFano}

\end{figure}
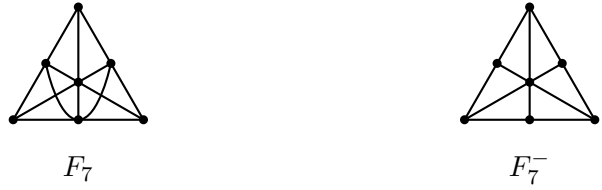

The Fano plane is representable only over fields of characteristic
two, and $F_{7}^{-}$ is representable only over fields of
characteristic other than two~\cite[p.~505]{oxley}.
Moreover, any proper minor of $F_{7}$ or $F_{7}^{-}$ is
either regular, or a whirl (up to the addition of parallel points).
The next result follows immediately.

\begin{prop}
\label{F7F7-}
The matroids $F_{7}$, $F_{7}^{-}$, and their duals,
are excluded minors for the class of near-regular matroids.
\end{prop}

The affine geometry \ag{2, 3} is produced by deleting
a hyperplane from the projective geometry \pg{2,3}.
Figure~\ref{figAG23} shows a geometric representation of \ag{2,3}.
Up to isomorphism there is a unique matroid produced by
deleting an element from \ag{2, 3}.
We denote this matroid by \agde.
It is not difficult to see that the automorphism group
of \agde\ acts transitively upon the triangles of \agde.
It follows that up to isomorphism there is a unique
matroid produced by performing a \dy\ operation on \agde.
We shall denote this matroid by $\Delta_{T}(\agde)$.
Then $\Delta_{T}(\agde)$ is represented over $\GF(3)$ by the following
matrix.
\begin{displaymath}
\begin{bmatrix}
\rbox{0} & \rbox{1} & \rbox{1} & \rbox{1}\\
\rbox{1} & \rbox{-1} & \rbox{1} & \rbox{-1}\\
\rbox{1} & \rbox{1} & \rbox{1} & \rbox{0}\\
\rbox{1} & \rbox{-1} & \rbox{0} & \rbox{0}\\
\end{bmatrix}
\end{displaymath}
Obviously $\Delta_{T}(\agde)$ is self-dual.

\begin{figure}[htb]

\centering

\begin{pspicture}(-2,0)(2,3)

\SpecialCoor

\rput(0,1.5){
\psset{unit=0.75cm}

\pscircle*(-1,-1){0.06cm}
\pscircle*(0,-1){0.06cm}
\pscircle*(1,-1){0.06cm}
\pscircle*(-1,0){0.06cm}
\pscircle*(0,0){0.06cm}
\pscircle*(1,0){0.06cm}
\pscircle*(-1,1){0.06cm}
\pscircle*(0,1){0.06cm}
\pscircle*(1,1){0.06cm}

\pspolygon(-1,-1)(1,-1)(1,1)(-1,1)
\psline(0,-1)(0,1)
\psline(-1,0)(1,0)
\psline(-1,1)(1,-1)
\psline(-1,-1)(1,1)

\pscurve(-1,0)(0,1)(2.2;35)(1,-1)
\pscurve(0,-1)(1,0)(2.2;55)(-1,1)

\pscurve(0,1)(1,0)(2.2;-55)(-1,-1)
\pscurve(-1,0)(0,-1)(2.2;-35)(1,1)
}

\end{pspicture}

\caption{\ag{2, 3}.}

\label{figAG23}

\end{figure}

\begin{prop}
\label{AG23}
The matroids \agde, $(\agde)^{*}$, and
$\Delta_{T}(\agde)$ are excluded minors for the
class of near-regular matroids.
\end{prop}

\begin{proof}
Suppose that we obtain a representation of \ag{2, 3} by
deleting from \pg{2, 3} all points
$[x_{1}, x_{2}, x_{3}]^{T}$ on the hyperplane defined by
$x_{1} + x_{2} + x_{3} = 0$.
We then obtain a representation of \agde\ by deleting the
point $[1, 1, -1]^{T}$.
Thus \agde\ is represented over $\GF(3)$ by the following matrix.
\begin{displaymath}
\begin{bmatrix}
\rbox{0} & \rbox{1} & \rbox{1} & \rbox{-1} & \rbox{1}\\
\rbox{1} & \rbox{0} & \rbox{1} & \rbox{1}  & \rbox{-1}\\
\rbox{1} & \rbox{1} & \rbox{0} & \rbox{1}  & \rbox{1}
\end{bmatrix}
\end{displaymath}
If \agde\ is $\GF(5)$\dash representable, then by normalizing,
we can assume that it is represented over $\GF(5)$ by the
following matrix.
\begin{displaymath}
\begin{bmatrix}
\rbox{0} & \rbox{1} & \rbox{1} & \rbox{1} & \rbox{1}\\
\rbox{1} & \rbox{0} & \rbox{$a$} & \rbox{$b$}  & \rbox{$c$}\\
\rbox{1} & \rbox{1} & \rbox{0} & \rbox{$d$}  & \rbox{$e$}
\end{bmatrix}
\end{displaymath}
Here, $\{a, b, c, d, e\}$ are nonzero elements of $\GF(5)$.
By comparing subdeterminants, we see that $e=1$, and that
$c - a - e = 0$, so that $c = a + 1$.
Moreover, $b = d = c$, so
$b$ and $d$ are also equal to $a+1$.
Finally $ad + b - a = 0$.
This means that $a$ is a root of the polynomial
$x^{2} + x + 1$.
But there is no such root in $\GF(5)$, so we have a contradiction.
Therefore \agde\ is certainly not near-regular.

The automorphism group of \ag{2, 3} is transitive on pairs of
elements.
It follows that the automorphism group of \agde\ is transitive
on points.
Using this fact, it is not difficult to see that any
single-element deletion of \agde\ is isomorphic to $P_{7}$
(illustrated in Figure~\ref{fig2}).
Now $P_{7}$ is representable over every field of cardinality
at least three~\cite[Lemma~6.4.13]{oxley}, and is therefore
near-regular.

On the other hand, by again using the transitivity of
\agde\ we can see that contracting any element from
\agde\ produces a matroid that is obtained from $U_{2,4}$
by adding parallel elements.
Thus every proper minor of \agde\ is near-regular, so
\agde\ is indeed an excluded minor for the class of
near-regular matroids.
It follows immediately that $(\agde)^{*}$ is an excluded
minor for the same class, and Lemma~\ref{lem4} implies that
$\Delta_{T}(\agde)$ is also an excluded minor for
near-regularity.
\end{proof}

The matroid $P_{8}$ is represented
over $\GF(3)$ by the following matrix:
\begin{displaymath}
\begin{bmatrix}
\rbox{0} & \rbox{1} & \rbox{1} & \rbox{-1}\\
\rbox{1} & \rbox{0} & \rbox{1} & \rbox{1}\\
\rbox{1} & \rbox{1} & \rbox{0} & \rbox{1}\\
\rbox{-1} & \rbox{1} & \rbox{1} & \rbox{0}
\end{bmatrix}
\end{displaymath}
The matroid $P_{8}''$ is obtained from $P_{8}$ by
relaxing its two circuit-hyperplanes.
Lemma~6.4.14 in~\cite{oxley} says that $P_{8}$ is
representable over a field if and only if
its characteristic is not two.
Thus $P_{8}$ is not near-regular.
However, every single-element deletion or contraction
of $P_{8}$ is isomorphic to either $P_{7}$ or
$P_{7}^{*}$~\cite[p.~513]{oxley}, and $P_{7}$ is
representable over every field containing at least three
elements.
The next result follows.

\begin{prop}
\label{lem6}
The matroid $P_{8}$ is an excluded minor for the class
of near-regular matroids.
\end{prop}

\subsection{Case-analysis}
Next we show that the list of excluded minors in
Theorem~\ref{thm1} is complete.
The matroids $P_{7}$ and $O_{7}$ are shown in
Figure~\ref{fig2}.

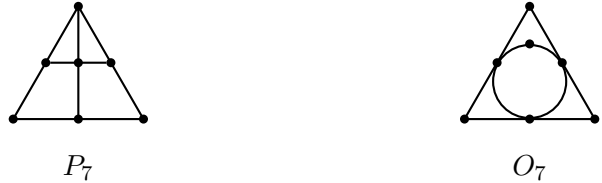
\begin{figure}[htb]

\centering

\begin{pspicture}(-3.86,0)(3.86,2.2)

\SpecialCoor

\rput(-3,1.25){
\pscircle*(1;90){0.06cm}
\pscircle*(1;210){0.06cm}
\pscircle*(1;330){0.06cm}
\pscircle*(0.5;270){0.06cm}

\pspolygon(1;90)(1;210)(1;330)
\psline(1;90)(0.5;270)

\cnode*(0.5;150){0.06cm}{a}
\cnode*(0.5;30){0.06cm}{b}

\ncline{a}{b}
\ncput{\pscircle*(0,0){0.06cm}}
}

\rput[B](-3,0){$P_{7}$}

\rput(3,1.25){
\pscircle*(1;90){0.06cm}
\pscircle*(1;210){0.06cm}
\pscircle*(1;330){0.06cm}
\pscircle*(0.5;270){0.06cm}
\pscircle*(0.5;150){0.06cm}
\pscircle*(0.5;30){0.06cm}
\pscircle*(0.5;90){0.06cm}

\pscircle(0,0){0.5cm}

\pspolygon(1;90)(1;210)(1;330)
}

\rput[B](3,0){$O_{7}$}

\end{pspicture}

\caption{$P_{7}$ and $O_{7}$.}

\label{fig2}

\end{figure}

The following matrix represents $O_{7}$ over any field
\field\ such that $|\field| \geq 3$.
Here, $\beta\in \field \setminus \{0,1\}$ if \field\ has characteristic
equal to two, and $\beta = -1$ otherwise.
\begin{displaymath}
\begin{bmatrix}
\rbox{1} & \rbox{1} & \rbox{0} & \rbox{1}\\
\rbox{1} & \rbox{0} & \rbox{1} & \rbox{-1}\\
\rbox{0} & \rbox{1} & \rbox{-1} & \rbox{$\beta$}
\end{bmatrix}
\end{displaymath}
It follows that $O_{7}$ is near-regular.
We have already noted that $P_{7}$ is near-regular.

\begin{prop}
\label{P7O7}
Let $M$ be a $3$\dash connected single-element
extension of $\mcal{W}^{3}$, the rank\dash $3$ whirl, such that
$M$ has no $U_{2,5}$\dash minor.
Then $M$ is isomorphic to one of $F_{7}^{-}$, $P_{7}$,
or $O_{7}$.
\end{prop}

\begin{proof}
Suppose that $M \del e$ is isomorphic to
$\mcal{W}^{3}$.
Let $E(M \del e) = \{r_{1},r_{2},r_{3},s_{1},s_{2},s_{3}\}$
and suppose that the triangles of
$M \del e$ are
$\{r_{1},s_{2},s_{3}\}$, $\{r_{2},s_{1},s_{3}\}$, and
$\{r_{3},s_{1},s_{2}\}$.
It is easy to see that if $e$ is contained in a four-point
line of $M$, then $M \cong O_{7}$.
Thus we assume $M$ contains no four-point lines.
But $e$ must be contained in a triangle with each of
$r_{1}$, $r_{2}$, and $r_{3}$, for otherwise
$M$ has a $U_{2,5}$\dash minor.
Now the result follows easily.
\end{proof}

\begin{lem}
\label{rank3}
Suppose that $M$ is an excluded minor for the
class of near-regular matroids, and that $\rank(M)=3$.
Then $M$ is isomorphic to one of $U_{3,5}$,
$F_{7}$, $F_{7}^{-}$, or \agde.
\end{lem}

\begin{proof}
Suppose that $M$ is a rank\dash $3$ excluded minor other than
those listed in the statement of the lemma.
Then $M$ must be ternary, for otherwise it contains
$U_{2,5}$, $U_{3,5}$, or $F_{7}$ as a minor~\cite{Bix79,Sey79}.
Since $M$ is not near-regular, and hence not regular, it
is nonbinary.
Certainly $M$ has at least six elements, and hence corank
at least three, for otherwise the fact that $M \ncong U_{3,5}$
means that $M$ is not $3$\dash connected.
Now Corollary~11.2.19 in~\cite{oxley}, and the fact that
$M$ has no $U_{2,5}$\dash minor, means that
$M$ has a $\mcal{W}^{3}$\dash minor.
Since $M$ is not isomorphic to \agde\ or
its dual, Theorem~\ref{thm:finitesize} implies that
$\corank(M) \leq 4$, and that therefore, $|E(M)| \leq 7$.
As $M$ is not isomorphic to $\mcal{W}^{3}$, it follows
that $M$ is a single-element extension of $\mcal{W}^{3}$.
Proposition~\ref{P7O7} implies that $M$ is isomorphic
to either $P_{7}$ or $O_{7}$.
As these are both near-regular we have a contradiction.
\end{proof}

Now we complete the proof of Theorem~\ref{thm1}.

\begin{thm1}
The excluded minors for the class of near-regular matroids are
$U_{2,5}$, $U_{3,5}$, $F_{7}$, $F_{7}^{*}$, $F_{7}^{-}$, \nfd, \agde,
$(\agde)^{*}$, $\Delta_{T}(\agde)$, and $P_{8}$.
\end{thm1}

\begin{proof}
The results in Section~\ref{excludedminors}
certify that the matroids listed in the theorem
are indeed excluded minors for near-regularity.
Now we suppose that $M$ is
an excluded minor and that $M$ is not listed in the statement
of the theorem.
Clearly the rank and corank of $M$ both exceed two.
Lemma~\ref{rank3} implies that they both exceed three.
It now follows from Theorem~\ref{thm:finitesize} that
both are exactly equal to four, so $M$ has precisely
eight elements.

Suppose that $M$ contains a triangle $T$.
As $M$ is $3$\dash connected, $T$ is coindependent.
Lemmas~\ref{DYrank} and~\ref{lem4} imply that
$\Delta_{T}(M)$ is an excluded minor
for near-regularity with corank three.
Now Lemma~\ref{rank3} implies that
$\Delta_{T}(M)$ is either $U_{2,5}$,
$F_{7}^{*}$, \nfd, or $(\agde)^{*}$.
As $M$ contains eight elements, we conclude that
$\Delta_{T}(M) \cong (\agde)^{*}$.
But $T$ is an independent triad in $\Delta_{T}(M)$,
by Lemma~\ref{DYrank}, and
\begin{multline*}
M=\nabla_{T}(\Delta_{T}(M)) \cong \nabla_{T}((\agde)^{*}) =\\
(\Delta_{T}(\agde))^{*} \cong \Delta_{T}(\agde).
\end{multline*}
This contradiction means that $M$ has no triangles.
The dual argument shows that $M$ has no triads.

As in the proof of Lemma~\ref{rank3}, we see that
$M$ is ternary and nonbinary, and that therefore
$M$ has a $\mcal{W}^{3}$\dash minor.
Since $M$ does not have a $\mcal{W}^{4}$\dash minor, we may
apply the Splitter Theorem.
By exploiting duality, we can assume that there are
elements $e,f \in E(M)$, such that $M \con e$ is
$3$\dash connected, and $M \con e \del f$ is isomorphic to
$\mcal{W}^{3}$.
Therefore $M \con e$ is isomorphic to $P_{7}$ or $O_{7}$,
by Lemma~\ref{P7O7}.

Assume that $M \con e \cong O_{7}$.
Since $M \con e$ contains a four-point line, and
$M$ contains no triangles, it follows that
$M$ contains a $U_{3,5}$\dash restriction.
This is a contradiction, so we assume that
$M \con e \cong P_{7}$.
By scaling, and uniqueness of representations, we can assume
that $M \con e$ is represented over $\GF(3)$ by the following
matrix.
\begin{displaymath}
\kbordermatrix{
& a & b & c & d\\
x & \rbox{0} & \rbox{1} & \rbox{1} & \rbox{-1}\\
y & \rbox{1} & \rbox{0} & \rbox{1} & \rbox{1}\\
z & \rbox{1} & \rbox{1} & \rbox{0} & \rbox{1}\\
e & \rbox{-1} & \rbox{$\alpha$} & \rbox{$\beta$} & \rbox{$\gamma$}
}
\end{displaymath}
The fact that $M$ contains no triangles means that
$\alpha$ and $\beta$ are nonzero, and that
$\gamma \ne -1$.
Moreover, $\alpha+\beta-\gamma \ne 0$.
This leaves us with five cases to check:
\begin{enumerate}
\item $\alpha = 1$, $\beta = 1$, $\gamma = 0$;
\item $\alpha = 1$, $\beta = 1$, $\gamma = 1$;
\item $\alpha=1$, $\beta=-1$, $\gamma=1$;
\item $\alpha=-1$, $\beta=1$, $\gamma=1$; and
\item $\alpha=-1$, $\beta=-1$, $\gamma=0$.
\end{enumerate}

If case~(i) holds, then we immediately see that $M$ is
isomorphic to $P_{8}$, a contradiction.
Suppose that~(ii) holds.
Then $M\con y \cong F_{7}^{-}$, a contradiction.
If~(iii) or~(iv) holds, then $M \cong P_{8}$.
Finally, if~(v) holds, then $M\con x \cong F_{7}^{-}$.
This contradiction completes the proof.
\end{proof}

\section{Acknowledgements}

Our thanks go to Jim Geelen and Geoff Whittle for
their helpful advice and discussions.

\end{document}